\documentclass[a4paper,11pt]{amsart}
\usepackage[matrix,arrow,tips,curve]{xy}
\usepackage[english]{babel}
\usepackage{amsmath}
\usepackage{amssymb}
\usepackage{pgf,tikz}
\usepackage{mathrsfs}
\usepackage{enumerate}
\usepackage{hyperref}

\input{xy}
\xyoption{all}
\newtheorem{thm}{Theorem}[section]
\newtheorem{Lemma}[thm]{Lemma}
\newtheorem{Proposition}[thm]{Proposition}
\newtheorem{Corollary}[thm]{Corollary}

\newtheorem*{thm*}{Theorem}

\theoremstyle{definition}

\newtheorem{Definition}[thm]{Definition}
\newtheorem{Remark}[thm]{Remark}

\newtheorem{Assumption}[thm]{Assumptions}
\newtheorem{say}[thm]{}

\newcommand{\arXiv}[1]{\href{http://arxiv.org/abs/#1}{arXiv:#1}}

\newcommand{\Aut}{\operatorname{Aut}}

\newcommand{\Ker}{\operatorname{Ker}}

\newcommand{\im}{\operatorname{Im}}

\newcommand{\Pic}{\operatorname{Pic}}

\newcommand{\codim}{\operatorname{codim}}
\newcommand{\Proj}{\operatorname{Proj}}
\newcommand{\Sing}{\operatorname{Sing}}
\newcommand{\p}{\mathbb{P}}
\newcommand{\cC}{\mathcal{C}}

\newcommand{\sD}{\mathscr{D}}
\newcommand{\sE}{\mathscr{E}}
\newcommand{\sF}{\mathscr{F}}

\newcommand{\sL}{\mathscr{L}}

\newcommand{\sO}{\mathscr{O}}

\newcommand{\C}{\mathbb{C}}
\newcommand{\bQ}{\mathbb{Q}}

\def\bibaut#1{{\sc #1}}
\renewcommand{\sec}{\mathbb{S}ec}

\renewcommand{\a}{\`a }

\newcommand{\Rat}{\textup{RatCurves}}

\begin{document}
\title{Codimension one Fano distributions on Fano manifolds}
\author[Carolina Araujo]{Carolina Araujo}
\address{\sc Carolina Araujo\\
IMPA\\
Estrada Dona Castorina 110\\
22460-320 Rio de Janeiro\\ Brazil}
\email{caraujo@impa.br}

\author[Mauricio Corr\^ea]{Mauricio Corr\^ea}
\address{\sc Mauricio Corr\^ea\\
UFMG\\
Avenida Ant\^onio Carlos, 6627\\
30161-970 Belo Horizonte\\ Brazil}
\email{mauriciojr@ufmg.br}

\author[Alex Massarenti]{Alex Massarenti}
\address{\sc Alex Massarenti\\
UFF\\
Rua M\'ario Santos Braga\\
24020-140, Niter\'oi,  Rio de Janeiro\\ Brazil}
\email{alexmassarenti@id.uff.br}

\date{\today}

\begin{abstract}
In this paper we investigate codimension one Fano distributions on Fano manifolds with Picard number one. 
We classify Fano distributions of maximal index on complete intersections in weighted projective spaces, Fano contact manifolds, Grassmannians of lines and their linear sections and describe their moduli spaces.
As a consequence, we obtain a classification of codimension one del Pezzo distributions on Fano manifolds with Picard number one.
\end{abstract}
\maketitle 

\setcounter{tocdepth}{1}

\tableofcontents

%
%

\section{Introduction}

Holomorphic distributions and foliations appear frequently in the study of complex projective manifolds. 
In recent years, foliations  with ample anti-canonical class, known as \emph{Fano foliations},  have been much investigated,
and those with most positive  anti-canonical class have been classified (see \cite{AD1}, \cite{AD2}, \cite{AD3} and \cite{AD4}). 
It is natural to aim at a similar classification for \emph{Fano distributions}.
In this paper we address a first instance of this problem: we investigate codimension one Fano distributions on Fano manifolds with Picard number one. 

Given a  holomorphic distribution $\sD\subset T_X$ on a complex projective manifold  $X$,  we define its canonical class to be $K_{\sD} = -c_1(\sD)$.
We say that $\sD$ is a \emph{Fano distribution} if $-K_{\sD}$ is ample, and in this case we define its \emph{index} $\iota_{\sD}$  to be 
the largest integer dividing $-K_{\sD}$ in $\Pic(X)$.
By \cite[Theorem 1.1]{ADK}, the index of a Fano distribution $\sD$ on a  complex projective manifold
is bounded above by its rank, $\iota_{\sD}\leq r_{\sD}$, and equality holds only if $X\cong \p^n$. 
Foliations $\sF$ on $\p^n$ having maximal index $\iota_{\sF}= r_{\sF}$ were classified in \cite[Th\'eor\`eme 3.8]{cerveau_deserti}:
they are induced by  linear projections $\p^n \dashrightarrow \p^{n-r_{\sF}}$.
We start by giving a similar classification for codimension one distributions $\sD$ on $\p^n$ having maximal index $\iota_{\sD}= r_{\sD}=\dim(X)-1$.
In order to state this, we need to introduce the \emph{class} of a  codimension one distribution. This is an invariant that measures how 
far $\sD$ is from being integrable. 

\begin{Definition}
Let $\sD\subset T_X$ be a codimension one distribution on a complex projective manifold  $X$, and consider its normal 
line bundle $\sL_\sD:=T_X/\sD$. 
The distribution $\sD$ corresponds to a unique (up to scaling) twisted $1$-form 
$\omega_{\sD}\in H^0(X,\Omega^1_X\otimes \sL_\sD)$  non vanishing in codimension one. 
This form uniquely determines the distribution $\sD$.
For every integer $i\geq 0$, there is a well defined twisted $(2i+1)$-form 
$$
\omega_{\sD}\wedge (d\omega_{\sD})^i\ \in \ H^0\Big(X,\Omega^{2i+1}_X\otimes \sL_\sD^{\otimes (i+1)}\Big).
$$
The \emph{class} of  $\sD$ is the unique non negative  integer $k=k({\sD})$ such that 
$$
\omega\wedge (d\omega)^k\neq0\  \ \text{ and } \ \ \omega\wedge (d\omega)^{k+1}=0.
$$
By Frobenius theorem, a codimension one distribution is a foliation if and only if $k({\sD})=0$. 
(See Section~\ref{section:distributions} for more details, including local normal forms for  class $k$ codimension one distributions.) 
\end{Definition}

\begin{say}[Distributions on projective spaces]\label{deg0_on_Pn}
When the ambient space is $\p^n$, a classical invariant of a codimension one distribution $\sD\subset T_{\p^n}$ is its \emph{degree}, defined as 
the number of tangencies of a general line with $\sD$. The degree $\deg(\sD)$ of $\sD$ is related to the index $\iota_{\sD}$ by the formula
$\deg(\sD)=n-1-\iota_{\sD}$. So distributions of degree zero on $\p^n$ are precisely those with maximal index. 

By Proposition \ref{weighted fol}, if $\sD$ is a degree zero codimension one distribution on $\p^n$ of class $k$ then up to change of coordinates, 
the associated form $\omega_{\sD}\in \big(H^0(\p^n,\Omega^1_{\p^n}(2))$ writes as:
$$
\omega_{\sD} \ = \ \sum_{i=0}^k (z_{2i}dz_{2i+1}-z_{2i+1}dz_{2i}).
$$

The projective space $\p\big(H^0(\p^n,\Omega^1_{\p^n}(2))\big)$ can be viewed as a parameter space for degree zero codimension one distributions
on $\p^n$, and it admits a stratification according to the class, which we now describe.
First we identify $H^0(\p^n,\Omega^1_{\p^n}(2))$ with $\bigwedge^2\C^{n+1}$. 
Let $D_k\subseteq\p\big(H^0(\p^n,\Omega^1_{\p^n}(2))\big)$ be the closed subset parametrizing distributions of class $\leq k$, with $0\leq k\leq \lfloor\frac{n-1}{2}\rfloor$.
Then, by Theorem \ref{mainstrat} the stratification 
$$
D_0\subseteq D_1\subseteq ...\subseteq D_{k-1}\subseteq ...\subseteq \p\big(H^0(\p^n,\Omega^1_{\p^n}(2))\big)
$$
corresponds to the natural stratification
$$
\mathbb{G}(1,n)\subseteq \sec_2(\mathbb{G}(1,n))\subseteq ...\subseteq \sec_k(\mathbb{G}(1,n))\subseteq ...\subseteq \mathbb{P}(\bigwedge^2\C^{n+1}),
$$
where $\sec_i(\mathbb{G}(1,n))$ is the $i^{th}$-secant variety of $\mathbb{G}(1,n)$ embedded by Pl\"ucker in $\mathbb{P}(\bigwedge^2\C^{n+1})$.
Note that the identification of $D_0$ with $\mathbb{G}(1,n)$ is natural from the classification of   degree zero codimension one foliations
on $\p^n$. Indeed, each such foliation is induced by a linear projection $\p^n \dashrightarrow \p^{1}$, i.e. by a pencil of hyperplanes in $\p^n$, 
i.e. by a line in $(\p^n)^{\vee}$. 

We refer to \cite{CCJ16} for a description of spaces of codimension one distributions of class one and low degree on $\mathbb{P}^3$
in terms of moduli spaces of stables sheaves. 
\end{say}

In this paper we extend the classification and the description of the parameter space in Paragraph~\ref{deg0_on_Pn} to a larger class
of Fano manifolds with Picard number one. More precisely, let  $X$ be a Fano manifold with Picard number one and index $\iota_X$, and 
write $A$ for the ample generator  of $\Pic(X)$. Our goals are the following.
\begin{enumerate} 
	\item[-] Find an effective upper bound $\iota_{max}(X)$ for the index of a codimension one Fano distribution $\sD$ on $X$.
	\item[-] Classify those $\sD$ attaining this bound, according to their class. 
	\item[-] Describe the stratification of the parameter space of such distributions
		$$
		\p H^0\Big(X,\Omega^1_{X}\big(\big(\iota_X-\iota_{max}(X)\big)A\big)\Big)
		$$
		given by the class. 
\end{enumerate} 

Our first result is the following general bound. We refer to Section~\ref{section:Fano_distributions} for the notion of 
minimal dominating family of rational curves.

\begin{Proposition}\label{bound_on_index}
Let $X$ be a Fano manifold with $\rho(X)=1$ and index $\iota_X$, and  write $A$ for the ample generator of $\Pic(X)$.
Let $\sD$ be a codimension one Fano distribution on $X$. Then:
\begin{enumerate}
	\item $\iota_{\sD}\leq \iota_X-1$.
	\item Assume moreover that $X$ admits a minimal dominating family of rational curves
		having degree one with respect to $A$ and whose general member is
		not tangent to $\sD$.  Then  $\iota_{\sD}\leq \iota_X-2$.
\end{enumerate} 
\end{Proposition}

The bound $\iota_{\sD}\leq \iota_X-1$ in Proposition \ref{bound_on_index} (1) is sharp for Fano contact manifolds. 
In this case, there is a unique distribution $\sD$ on $X$ attaining this bound, namely the contact structure on $X$
(see Proposition~\ref{Prop:contact}).

When the bound $\iota_{\sD}\leq \iota_X-2$ in Proposition \ref{bound_on_index} (2) is attained, the distribution $\sD$ 
is defined by a twisted $1$-form $\omega_{\sD}\in  H^0\big(X,\Omega^1_X(2A)\big)$. 
We show that this holds for complete intersections in projective spaces, and that distributions 
of maximal index are precisely those induced by the ones in the ambient space. 
More precisely, we have the following classification and description of the parameter space. 

\begin{thm}\label{stratwps}
Let $X\subset \mathbb{P}^n$ be a smooth complete intersection. Then
\begin{enumerate}

\item $H^0(X,\Omega^1_{X}(1))=0$.

\item Let 
$$
D_k\subseteq\mathbb{P}(H^0(\mathbb{P}^n,\Omega^1_{\mathbb{P}^n}(2)))
$$ 
be the subvariety parametrizing  distributions of class $\leq k$ on $\mathbb{P}^n$, and let 
$$
\overline{D}_k\subseteq\mathbb{P}(H^0(X,\Omega^1_{X}(2)))
$$ 
be the subset parametrizing distributions of class $\leq k$ on $X$.

Then there is a natural restriction isomorphism $H^0(\mathbb{P}^n,\Omega_{\mathbb{P}^n}^{1}(2))\cong H^0(X,\Omega_{X}^{1}(2))$
that maps $D_k$ isomorphically onto $\overline{D}_k$ for any $k < \lfloor\frac{\dim(X)-1}{2}\rfloor$. 
\end{enumerate}
\end{thm}

In fact, Theorem \ref{stratwps} is a special case of Theorem \ref{main-ci-wps}, which deals with complete intersections in weighted projective spaces.

Next we turn our attention to Fano manifolds of high index.
Fano manifolds of dimension $n$ and index $\iota_X\geq n-2$ have been classified.
By \cite{kobayashi_ochiai}, $\iota_X\leq n+1$, and equality  holds if and only if $X\cong \p^n$. 
Moreover, $\iota_X= n$ if and only if $X$ is a  quadric hypersurface $Q^n\subset \p^{n+1}$. 
These two cases are addressed in Paragraph~\ref{deg0_on_Pn} and Theorem~\ref{stratwps}, respectively.
Fano manifolds with index $\iota_X= n-1$ are called \emph{del Pezzo} manifolds, and were classified by Fujita in \cite{fujita82a} and  \cite{fujita82b}.
The ones with $\rho(X)=1$ 
are isomorphic to one of the following.
\begin{enumerate}
	\item A cubic hypersurface $X_3\subset\p^{n+1}$ with $n\geq 3$.
	\item An intersection of two quadric hypersurfaces in $X_{2,2}\subset\p^{n+2}$ with $n\geq 3$.
	\item A hypersurface of degree $4$ in the weighted projective space $X_4\subset\p(1,1,\ldots,1,2)$ with $n\geq 3$. Alternatively, $X_4$ is a double cover of $ \p^n$ branched along a quartic.
	\item A hypersurface of degree $6$ in the weighted projective space $X_6\subset\p(1,\ldots,1,2,3)$ with $n\geq 2$. Alternatively, $X_6$ is a double cover of $\p(2,1,\ldots,1)$ branched along a sextic.
	\item A linear section of codimension $c\leq 3$ of of the Grassmannian $\mathbb{G}(1,4)\subset\p^9$ under the Pl\"ucker embedding.
\end{enumerate}

Fano manifolds with $\iota_X= n-2$ are called \emph{Mukai} manifolds.
Their classification was first announced in \cite{Mukai89}.
We refer to  \cite[Theorem 7]{Araujo_Castravet2} for the full list of Mukai manifolds with  $\rho(X)=1$.
For del Pezzo and Mukai manifolds we have the following results. 

\begin{thm}\label{bound_for_dP_Mukai}
Let $X$ be an $n$-dimensional Fano manifold with $\rho(X)=1$, and $\sD$  a codimension one Fano distribution on $X$.
\begin{enumerate}
\item If $\iota_X = n-1$, then $n\geq 4$ and $\iota_{\sD}\leq n-3$.
\item If $\iota_X = n-2$ and $n\geq 6$, then $\iota_{\sD}\leq n-4$.
\end{enumerate}
\end{thm}

The assumption $n\geq 6$ in Theorem \ref{bound_for_dP_Mukai} (2) is indeed necessary, as Example~\ref{G2} illustrates.  

The bound in Theorem \ref{bound_for_dP_Mukai} (1) is sharp. For foliations, this bound is attained precisely by foliations induced by a pencil of hyperplane sections in 
$\big| A\big|$ \cite[Theorem 5]{AD2}. In Theorems \ref{stratwps} and \ref{main-ci-wps} we classify and describe codimension one distributions of arbitrary class attaining this bound for del Pezzo manifolds $(1)-(4)$ above. For Grassmannians of lines and their linear sections we have the following.

\begin{thm}\label{main_G}
Let $\mathbb{G}(1,n)\subset\mathbb{P}^N$ be the Grassmannian of lines in $\mathbb{P}^n$ embedded via the Pl\"ucker embedding, 
and let $X_i = \mathbb{G}(1,n)\cap H_1\cap ...\cap H_i$ be a codimension $i$ smooth linear section of $\mathbb{G}(1,n)$, with $0\leq i \leq 2(n-1)-4$.
Then the following hold.
\begin{enumerate}
\item $H^0(X_i,\Omega_{X_i}^1(1))=0$.
\item The restriction map
$
r:H^0(\mathbb{P}^N,\Omega_{\mathbb{P}^N}^1(2))\rightarrow H^0(\mathbb{G}(1,n),\Omega_{\mathbb{G}(1,n)}^1(2))
$
is an isomorphism.
\item The restriction map 
$
r_i:H^0(X_{i-1},\Omega_{X_{i-1}}^1(2))\rightarrow H^0(X_i,\Omega_{X_i}^1(2))
$
is surjective and corresponds to a linear projection
$
\pi_i:\mathbb{P}(H^0(X_{i-1},\Omega_{X_{i-1}}^1(2)))\dashrightarrow \mathbb{P}(H^0(X_i,\Omega_{X_i}^1(2)))
$
with center $L\cong \mathbb{P}(H^0(X_i,\mathcal{O}_{X_i}(1)))$. 

The distributions in the center $L\subset \mathbb{P}(H^0(X_{i-1},\Omega_{X_{i-1}}^1(2)))$
are integrable, and are induced by  linear projections 
$H_1\cap ...\cap H_{i-1}\cong \mathbb{P}^{N-i+1}\dasharrow\mathbb{P}^1$ from codimension two linear subspaces contained in $H_i$.
\item ($n=4$) Consider the restriction map
$
r:H^0(\mathbb{P}^9,\Omega_{\mathbb{P}^9}^1(2))\rightarrow H^0(\mathbb{G}(1,4),\Omega_{\mathbb{G}(1,4)}^1(2))
$,
and let $\omega\in H^0(\mathbb{P}^9,\Omega_{\mathbb{P}^9}^1(2))$. Then 
\begin{enumerate}
\item If $r(\omega)$ has class zero, then $\omega$ has class zero.
\item If $r(\omega)$ has class one, then one of the following holds:
\begin{itemize}
\item[-]  $\omega$ has class one, or
\item[-] $\omega$ has class two, and the characteristic foliation of $\omega$, induced by $\omega\wedge (d\omega)^2$, 
is the linear projection $\p^9\dasharrow \p^5$ from a $3$-dimensional linear subspace contained in $\mathbb{G}(1,4)$.
\end{itemize}
\end{enumerate}
\end{enumerate}
\end{thm}

As in the case of Fano foliations, we say that a Fano distribution $\sD$ is \textit{del Pezzo} if $\iota_{\sD} = r_{\sD}-1$. As a consequence of the above results we classify codimension one del Pezzo distributions on Fano manifolds with Picard number one.

\begin{Proposition}\label{th1}
Let $\sD$ be a codimension one del Pezzo distribution on a Fano manifold $X$ of dimension $n$ with $\rho(X)=1$. Then the pair $(X,\sD)$ satisfies one of the following conditions:
\begin{enumerate}
	\item[-] $X\cong \p^n$ and $\sD$ is a distribution of degree one;
	\item[-] $X\cong Q^n\subset \p^{n+1}$ and $\sD$ is the restriction of a degree zero distribution on $\mathbb{P}^{n+1}$.
\end{enumerate}
\end{Proposition}

This paper is organized as follows. In section~\ref{section:distributions}, we introduce holomorphic distributions and foliations on complex projective varieties, and collect 
some of their basic properties. 
In section~\ref{section:Fano_distributions}, we turn our attention to Fano distributions, and prove general bounds 
for the index of codimension one Fano distributions on Fano manifolds with Picard number one. 
We address distributions of maximal index on weighted projective spaces and
on complete intersections in them in sections~\ref{section:P^n} and \ref{section:wci}, respectively. 
In section~\ref{section:G}, we discuss  distributions of maximal index on Grassmannians of lines and their linear sections.

\subsection*{Notation and Conventions}
We always work over the field ${\mathbb C}$ of complex numbers. 
Given a normal variety $X$, we  denote by $T_X$ the sheaf $(\Omega_{X}^1)^*$.

\subsection*{Acknowledgments}
The first named author was partially supported by CNPq and Faperj Research  Fellowships. The second  named author was partially supported by CAPES, CNPq and Fapesp-2015/20841-5 Research  Fellowships.
The third named author is a member of the Gruppo Nazionale per le Strutture Algebriche, Geometriche e le loro Applicazioni of the Istituto Nazionale di Alta Matematica "F. Severi" (GNSAGA-INDAM). We thank Jos\'e Carlos Sierra for useful discussions about the projective geometry of Grassmannians.

%
%

\section{Holomorphic distributions} \label{section:distributions}

In this section we present some basic facts about holomorphic distributions and foliations on complex projective varieties. 
Throughout this section, unless otherwise noted, $X$ denotes a normal variety of dimension $n\geq 2$. 

\begin{Definition}
A \emph{(holomorphic) distribution} on $X$ is a  nonzero subsheaf $\sD\subset T_X$ which is saturated, i.e. such that the quotient $T_X/\sD$ is torsion-free.

The singular locus of $\sD$ is the locus $\Sing(\sD)$ where $T_X/\sD$ fails to be locally free. 

The \emph{rank} $r_\sD$ of $\sD$ is the generic rank of $\sD$.
The \emph{codimension} of $\sD$ is defined as $q:=\dim X-r_\sD$. 

The \emph{normal sheaf} of $\sD$ is the reflexive sheaf $N_\sD:=(T_X/\sD)^{**}$. We denote its determinant by $\sL_{\sD}=\det(N_\sD)$.

The \textit{canonical class} $K_{\sD}$ of $\sD$ is any Weil divisor on $X$ such that  $\sO_X(-K_{\sD})\cong \det(\sD)$. 
%
\end{Definition}

\begin{say}[Pullback distributions]
Let $\varphi:X\dasharrow Y$ be a dominant rational map with connected fibers between normal varieties, and $\sD_Y$  a distribution on $Y$.
Let $X^\circ\subset X$ and  $Y^\circ\subset Y$ be smooth open subsets such that $\varphi$ restricts to a morphism $\varphi^\circ\colon X^\circ\to Y^\circ$.
Then there is a unique distribution $\sD_X$ on $X$ such that ${\sD_X}_{|X^\circ}=(d\varphi^\circ)^{-1}({\sD_Y}_{|Y^\circ})$.
We say that $\sD_X$ is the pullback of $\sD_Y$ by $\varphi$. 
\end{say}

\begin{say}[Distributions and differential forms]
Let $\sD\subset T_X$ be a codimension $q$ distribution on $X$.
The $q$-th wedge product of the inclusion $N^*_\sD\subset \Omega^1_X$ gives rise to a  twisted $q$-form (unique up to scaling)
$\omega_{\sD}\in  H^0(X,\Omega^q_X\otimes \sL_\sD)$ non vanishing in codimension $1$. 
This form locally decomposes as the wedge product of $q$ local $1$-forms at smooth points of $X\setminus \Sing(\sD)$, and
uniquely determines the distribution $\sD$.
More precisely, $\sD$ is the  kernel
of the morphism $T_X \to \Omega^{q-1}_X\otimes \sL_{\sD}$ given by the contraction with $\omega_{\sD}$. 
\end{say}

\begin{say}\label{frobenius}
By Frobenius' theorem, a  distribution $\sD\subset T_X$ is integrable, i.e. it is the tangent sheaf of a holomorphic  foliation, if and only if 
it is closed under the Lie bracket. 
In terms of the associated twisted $q$-form  $\omega_{\sD}\in H^0(X,\Omega^q_X\otimes \sL_\sD)$, this condition 
is equivalent to the following. 
If $\omega_{\sD}=\omega_1\wedge\cdots\wedge\omega_q$ is a local decomposition of $\omega_{\sD}$ 
as the wedge product of $q$ local $1$-forms, then it satisfies $d\omega_i\wedge \omega=0$ for every  $i\in\{1,\ldots,q\}$. 
When $\sD$ has codimension one, this  reduces to 
$$ 
\omega_\sD\wedge d\omega_\sD=0.
$$
By abuse of notation, when $\sD\subset T_X$ is integrable, we say that $\sD$ itself is a foliation. 
\end{say}

Next, for a  codimension one distribution $\sD$, we define the \emph{class} of $\sD$. This is an invariant that measures how 
far $\sD$ is from being integrable.

\begin{Definition}
Let $\sD\subset T_X$ be a codimension one distribution on $X$, and consider the associated twisted $1$-form 
$\omega_{\sD}\in H^0(X,\Omega^1_X\otimes \sL_\sD)$.
For every integer $i\geq 0$, there is a well defined twisted $(2i+1)$-form 
$$
\omega_{\sD}\wedge (d\omega_{\sD})^i\ \in \ H^0\big(X,\Omega^{2i+1}_X\otimes \sL_\sD^{\otimes (i+1)}\big).
$$
The \emph{class} of  $\sD$ is the unique non negative  integer $k=k({\sD})\in\big\{0, \cdots, \lfloor\frac{n-1}{2}\rfloor\big\}$ such that 
$$
\omega_{\sD}\wedge (d\omega_{\sD})^k\neq0\  \ \text{ and } \ \ \omega_{\sD}\wedge (d\omega_{\sD})^{k+1}=0.
$$
\end{Definition}

\begin{say}[Local description of a codimension one distribution of class $k$] 
Let $\sD\subset T_X$ be a codimension one distribution of class $k$ on $X$.
Then, at any smooth point  $x\in X\setminus \Sing(\sD)$, there are analytic local coordinates $(z_1, \dots, z_n)$
such that $\omega_{\sD}$ writes as 
\stepcounter{thm}
\begin{equation}\label{eq:local_form}
	\omega_{\sD} \ = \ dz_1 + \sum_{i=1}^k (z_{2i}dz_{2i+1}-z_{2i+1}dz_{2i})
\end{equation}
in an analytic neighborhood of $x$ (see \cite{BCGGG}).
\end{say}

\begin{Definition}
Let $\sD\subset T_X$ be a codimension one distribution of class $k$ on $X$, and assume that $n>2k+1$.
From the normal form \eqref{eq:local_form}, one can check that  the twisted $(2k+1)$-form 
$$
\theta_{\sD} = \omega_{\sD}\wedge (d\omega_{\sD})^k \in H^0\big(X,\Omega^{2k+1}_X\otimes \sL_\sD^{\otimes (k+1)}\big)\setminus \{0\}
$$
satisfies the integrability condition discussed in Paragraph~\ref{frobenius}. 
Hence  $\theta_{\sD} $ induces a codimension $2k+1$ foliation $\cC h({\sD}) \subset T_X$,  the
\emph{characteristic foliation} of $\sD$. 
It  can be characterized as the subsheaf of $\sD$ generated by all  germs of vector fields $v$ tangent to $\sD$ 
and satisfying $[v, \sD] \subset \sD$.
\end{Definition}

In an analytic neighborhood $U$ of a smooth point of $X\setminus \Sing(\sD)$ where $\omega_{\sD}$ is given by \eqref{eq:local_form},
the characteristic foliation $\cC h({\sD})$ corresponds to the projection 
\stepcounter{thm}
\begin{equation}\label{eq:char_foliation}
	\begin{aligned}
	U \ & \  \to \ \ \C^{2k+1} \\
	(z_1, \dots, z_n) \ & \ \mapsto  \ \ (z_1, \dots, z_{2k+1}).
	\end{aligned}
\end{equation}

As discussed in Paragraph \ref{frobenius} a codimension one distribution $\sD\subset T_X$
is a foliation if and only if $k({\sD})=0$. 
In the other extreme case, when $n$ is odd and $k({\sD})=\frac{n-1}{2}$,  $\sD$ is called a  \emph{contact distribution}.

\begin{Definition}
Let $X$ be smooth projective variety of odd dimension $n=2m+1\geq 3$.
A \emph{nonsingular contact structure} on $X$ is a codimension one distribution $\sD\subset T_X$ of maximal class $k=m$ on $X$
satisfying the following conditions:
\begin{enumerate}
	\item[-] $\Sing(\sD)=\emptyset$.
	\item[-] The the twisted $n$-form $\theta_{\sD}=\omega_{\sD}\wedge (d\omega_{\sD})^m  \in H^0\big(X,\Omega^n_X\otimes \sL_\sD^{\otimes (m+1)}\big)$ is nowhere vanishing.
\end{enumerate} 
The second condition implies that 
\stepcounter{thm}
\begin{equation}\label{eq:contact}
	-K_X \ = \ (m+1) \  c_1(\sL_{\sD}) \ .	
\end{equation}
\end{Definition}

%
%

\section{Fano distributions} \label{section:Fano_distributions}

In this section we address Fano distributions of high index on Fano manifolds of Picard number one.

\begin{Definition} 
Let $X$ be a normal projective variety, and $\sD\subset T_X$ a distribution. 
We say that $\sD$ is a \emph{Fano distribution} if its anti-canonical class $-K_{\sD}$ is an ample $\bQ$-Cartier divisor. 
The index of a Fano distribution $\sD$ is the largest rational number  $\iota_{\sD}$ such that $-K_{\sD} \sim_\bQ \iota_{\sD}A$
for a Cartier divisor $A$ on $X$.
\end{Definition}

From now on in this section, we let $X$ be a Fano manifold with index $\iota_X$ and $\rho(X)=1$, and 
write $A$ for the ample generator of $\Pic(X)$, so that $-K_X=\iota_XA$. 
We give  general bounds for the index of a Fano distribution on $X$ in terms of $\iota_X$.
The theory of rational curves on varieties proves useful in this context.

\begin{say}[Minimal rational curves and the variety of minimal rational tangents]
Let $H$ be a \emph{minimal dominating family of rational curves} on $X$, i.e.
$H$ is  an irreducible component of $\Rat(X)$ such that 
\begin{itemize}
	\item[-] curves parametrized by $H$ sweep out a dense subset of $X$, and
	\item[-] for a general point $x\in X$, the subset $H_x\subset H$ parametrizing curves through $x$ is proper.
\end{itemize}
The theory of minimal rational curves was initiated in \cite{Mo79}. Using his bend and break technique, Mori proved that a curve $\ell\subset X$ parametrized by $H$ satisfies 
\stepcounter{thm}
\begin{equation}\label{eq:degree_minimal_rat_curve}
-K_X\cdot \ell \leq \dim(X)+1.
\end{equation}
(See also \cite[IV.1.15]{kollar96}.)
For a general point $x\in X$, let $\tilde H_x$ be the normalization of $H_x$.
Then  $\tilde H_x$ is a finite union of smooth projective varieties of dimension equal to $-K_X\cdot \ell-2$
(see \cite[II.1.7, II.2.16]{kollar96}).  
The tangent map $ \tau_x: \ \tilde H_x \dashrightarrow  \p(T_xX) $ is defined by sending a curve that is smooth at $x$ to its
tangent direction at $x$.  The image $\cC_x\subset \p(T_xX)$  of
$\tau_x$ is called the \emph{variety of minimal rational tangents} at $x$ associated to  family  $H$.
The map $\tau_x: \ \tilde H_x \to \cC_x$ is the normalization morphism by \cite{kebekus02} and \cite{hwang_mok04}.  
\end{say}

\begin{proof}[{Proof of Proposition~\ref{bound_on_index}}]
Let $\omega_{\sD}\in H^0\big(X,\Omega^1_{X}\big((\iota_X-\iota_{\sD})A\big)\big)$ be a $1$-form associated to $\sD$. 

Let $\ell\subset X$ be a rational curve on $X$, not contained in the singular locus of $\sD$,  and not tangent to $\sD$. 
Denote by $f:\p^1\to \ell$ the normalization morphism, and set $a=A\cdot \ell\geq 1$. 
Then the  pullback of $\omega_{\sD}$ to $\p^1$ yields a nonzero twisted $1$-form 
$$
\omega \in H^0\big(\p^1,\omega_{\p^1}\big((\iota_X-\iota_{\sD})a\big)\big).
$$
So we must have $(\iota_X-\iota_{\sD})a\geq 2$, which implies that  $\iota_{\sD}\leq \iota_X-1$.

Now suppose that $X$ admits a minimal dominating family of rational curves $H$
having degree one with respect to $A$, and whose general member is not tangent to $\sD$.  
Then we may take the above curve $\ell$ to be a general curve parametrized by $H$.
In this case we get a nonzero $1$-form  $\omega \in H^0\big(\p^1,\omega_{\p^1}(\iota_X-\iota_{\sD})\big)$,
and thus $\iota_{\sD}\leq \iota_X-2$.
\end{proof}

\begin{Corollary} \label{corollay_high_index}
Let $X$ be a Fano manifold with $\rho(X)=1$, and $\sD$ a codimension one Fano distribution on $X$. 
Suppose that $\iota_X>\frac{\dim(X)+1}{2}$,  and that the variety of minimal rational tangents  is smooth
for some choice of a minimal dominating family of rational curves on $X$.
Then  $\iota_{\sD}\leq \iota_X-2$.
\end{Corollary}

\begin{proof}
Let $H$ be a minimal dominating family of rational curves on $X$, and $\ell\subset X$  a general curve parametrized by $H$.
Suppose that the associated variety of minimal rational tangents at a general point  $\cC_x\subset \p(T_xX)$ is smooth.
If a general curve parametrized by $H$ is tangent to $\sD$, then $\cC_x\subset \p(\sD_x)$ is degenerate in  $\p(T_xX)$. 
On the other hand, by \cite[Theorem 2.5]{Hwang_ICTP}, if $\iota_X>\frac{\dim(X)+1}{2}$ and $\cC_x$
is smooth, then it is non degenerate in  $\p(T_xX)$. 
This implies that a general curve parametrized by $H$  is not tangent to $\sD$. 
Moreover, the condition  that $\iota_X>\frac{\dim(X)+1}{2}$ together with \eqref{eq:degree_minimal_rat_curve} imply that $\ell$ 
has degree one with respect to $A$.
The result then follows from Proposition~\ref{bound_on_index}.
\end{proof}

\begin{Remark}\label{Cx_smooth}
Let $H$ be a minimal dominating family of rational curves on $X$.
Let $\cC_x\subset \p(T_xX)$ be the variety of minimal rational tangents  at a general point associated to $H$.
In general, $\cC_x$ may not be smooth. The first non smooth example of $\cC_x$ was given in \cite{Hwang_Kim_AG2015}
(see also \cite{Casagrande_Druel_2012}). 
On the other hand, it follows from the argument in the proof of \cite[Proposition 1.5]{Hwang_ICTP} that 
$\cC_x$ is smooth if the following condition holds. 

\center{\it{$X$ admits a finite morphism $\varphi:X\to \p^N$ such that curves parametrized by $H$ \\
are sent to lines in $\p^N$.}}
\end{Remark}

\begin{proof}[{Proof of Theorem \ref{bound_for_dP_Mukai}}]
Recall from the introduction the list of del Pezzo manifolds with $\rho(X)=1$.

The del Pezzo manifolds (1)-(3) and (5) satisfy  the condition in Remark~\ref{Cx_smooth}
for a  minimal dominating family of rational curves.
Therefore,  their associated variety of minimal rational tangents is smooth. 
For those manifolds, the result  follows from Corollary~\ref{corollay_high_index}. 

Let  $X$ be a del Pezzo manifold as in (4), i.e., a hypersurface of degree $6$ in the weighted projective space $\p(3,2,1,\ldots,1)$.
If $n\geq 6$, then the  variety of minimal rational tangents 
associated to a minimal dominating family of rational curves on $X$ is not smooth by \cite[Theorem 1.3]{Hwang_Kim_AG2015}.
But in this case the bound  follows from the vanishing in Lemma~\ref{lf1}:
$H^0\big(X, \Omega_X^1(A)\big)=0$.

From the list of Mukai manifolds with $\rho(X)=1$ in \cite[Theorem 7]{Araujo_Castravet2},
one can easily check that each Mukai manifold of dimension $n\geq 4$ in that list satisfies the condition in Remark~\ref{Cx_smooth}
for a  minimal dominating family of rational curves.
Therefore,  their associated variety of minimal rational tangents is smooth. 
The result then follows from Corollary~\ref{corollay_high_index}. 
\end{proof}

The assumption $\iota_X>\frac{\dim(X)+1}{2}$ in Corollary \ref{corollay_high_index}, 
and the assumption $n\geq 6$ in Theorem \ref{bound_for_dP_Mukai} are indeed necessary. This is illustrated by the case of Fano contact manifolds, which we now explain. 

\begin{say}[Fano contact manifolds]\label{contact_Fanos}
A Fano manifold $X$ of odd dimension $n=2m+1\geq 3$ together with a nonsingular contact structure $\sD$ on it is called a \emph{Fano contact manifold}. 
Let $(X,\sD)$ be a Fano contact manifold. By \cite{Wis} and \cite{LS}, equality (\ref{eq:contact}) implies that $X$ satisfies one of the following:
\begin{itemize}
\item[-] $\Pic(X) = \mathbb{Z}[\sL_{\sD}]$,
\item[-] $X\cong\mathbb{P}^n$ and $\sL_{\sD}\cong\mathcal{O}_{\mathbb{P}^n}(2)$,
\item[-] $X\cong\mathbb{P}(T_{\mathbb{P}^{m+1}})$.
\end{itemize}
From now on we assume that $\Pic(X) = \mathbb{Z}[\sL_{\sD}]$. 
In this case, \eqref{eq:contact} yields $\iota_X = m+1$, and since $\sL_{\sD} \cong\mathcal{O}_{X}(-K_X+K_{\sD})$, we have $\iota_{\sD}=m$.

By \cite{CMS-B} (see also \cite{keb}), there exists a minimal dominating family $H$ of rational curves on $X$ having degree one with respect to $\sL_{\sD}$. 
By \cite[Proposition 2]{Hwang-contact} the corresponding variety of minimal rational tangents at a general point $x\in X$ satisfies
$$
\mathcal{C}_x\subseteq\mathbb{P}(\sD_x)\subseteq\mathbb{P}(T_xX).
$$
Moreover,  $\mathbb{P}(\sD_x)$ is the linear span of $\mathcal{C}_x$.

\begin{Proposition}\label{Prop:contact}
Let $(X,\sD)$ be a Fano contact manifold of dimension $n=2m+1$, and $\Pic(X) = \mathbb{Z}[\sL_{\sD}]$.
Let $\sE\subseteq T_X$ be a Fano distribution on $X$. 
Then the index of $\sE$ satisfies $\iota_{\sE}\leq \iota_{X}-1=m$, and equality holds if and only if $\sE = \sD$.
\end{Proposition}
\begin{proof}
As noted in Paragraph~\ref{contact_Fanos}, $X$ admits a minimal dominating family of rational curves having degree one with respect to $\sL_{\sD}$. 
Since $\iota_X = m+1$, Proposition \ref{bound_on_index} yields that $\iota_{\sE}\leq m$. 
If $\iota_{\sE}=m$, then Proposition \ref{bound_on_index} implies that the general member of the family is tangent to $\sE$.
Since $\mathbb{P}(\sD_x)$ is the linear span of $\mathcal{C}_x$, it follows that $\sD = \sE$.
\end{proof}

\begin{Remark}\label{G2}
For the homogeneous contact manifold $(G_2,\sD)$ in \cite[Section 1.4.6]{Hwang_ICTP}, we have $\iota_{G_2}=3$ and $\iota_{\sD}=2$. 
Therefore, the assumption $\iota_X>\frac{\dim(X)+1}{2}$ in Corollary \ref{corollay_high_index}, and the  assumption $n\geq 6$ in 
Proposition \ref{bound_for_dP_Mukai} are indeed necessary. 
\end{Remark}
\end{say}

%
%

\section{Weighted projective spaces} \label{section:P^n}

In this section we address Fano distributions of maximal index on weighted projective spaces.

\begin{say}[Weighted projective spaces]\label{WPS0}
Let $a_0,\ldots, a_N$ be positive integers, and assume that $\gcd(a_0,\ldots, \hat a_i, \ldots a_N)=1$ for every $i\in\{0, \ldots, N\}$.
Denote by $S(a_0,\ldots, a_N)$ the polynomial ring ${\mathbb C}[z_0,\ldots, z_N]$ graded by $\deg z_i=a_i$, and
set $\p=\p(a_0,\ldots, a_N)=\Proj \big(S(a_0,\ldots, a_N)\big)$.
For each $t\in\mathbb{Z}$, let $\sO_{\p}(t)$ be the $\sO_{\p}$-module associated to the graded $S$-module $S(t)$.

From the Euler sequence for weighted projective spaces, it follows that a nonzero twisted $1$-form  $\omega \in  H^0( \p,\Omega_{ \p}^1(r))$
can be written as:
\stepcounter{thm}
\begin{equation}\label{Euler}
\omega \ =\ \sum_{i=0}^{N} F_idz_i, 
\end{equation}
with $F_i$ weighted homogeneous of degree $r-a_i$, and such that $\sum_{i=0}^{N}a_iz_iF_i=0$. 
\end{say}

\begin{Proposition}\label{weighted fol}
Let $\p=\p(a_0,\ldots, a_N)$ be as above, with $a_0 =a_1  =\cdots =a_\ell <  a_{\ell+1} \leq \cdots\leq a_N$, and let
$\varphi:\p\dasharrow \p^\ell$ be the rational map defined by $(z_0:\cdots: z_\ell)$.
Let $\sD$ be a codimension one distribution of class $k$ on $\p$, induced by a $1$-form $\omega \in  H^0( \p,\Omega_{ \p}^1(2a_0))$.
Then $k \leq  \lfloor\frac{\ell-1}{2}\rfloor$  and, 
up to linear change of coordinates in $\p^\ell$,
 $\sD$ is the pull-back via $\varphi$ of the distribution on $\mathbb{P}^{\ell}$ defined by 
$$
\sum_{i=0}^k (z_{2i}dz_{2i+1}-z_{2i+1}dz_{2i}) \in H^0( \p^{\ell},\Omega_{ \p^{\ell}}^1(2)).
$$
\end{Proposition}

\begin{proof}
By (\ref{Euler}), we can write the $1$-form $\omega \in  H^0( \p,\Omega_{\p}^1(2a_0))$ as 
$\omega = \sum_{0\leq i<j\leq \ell}a_{ij}(z_{i}dz_{j}-z_{j}dz_{i})$, with $a_{ij} \in \mathbb{C}$. 
Thus $d\omega = 2\sum_{0\leq i<j\leq \ell}a_{ij}dz_i\wedge dz_j$.  By changing coordinates we may write, for some integer $0\leq r\leq  \lfloor\frac{\ell-1}{2}\rfloor$, 
$d\omega = \sum_{i=0}^r dz_{2i}\wedge dz_{2i+1}$ and  $\omega = \sum_{i=0}^r (z_{2i}dz_{2i+1}-z_{2i+1}dz_{2i})$.  
Since $\omega$ is of class $k$, a straightforward computation gives that $r = k$.
\end{proof}

\begin{say}[Secant varieties]
Given a non degenerate variety $X\subset\mathbb{P}^N$, and a positive integer $k\leq N$ we denote by $\sec_k(X)$ 
the \emph{$k$-secant variety} of $X$. This is the subvariety of $\mathbb{P}^N$ obtained as the closure of the union of all $(k-1)$-planes 
$\langle x_1,...,x_{k}\rangle$ spanned by $k$ general points of $X$. 
We will be concerned with the case when $X=\mathbb{G}(1,n)$ is the Grassmannian of lines in $\mathbb{P}^n$.
\end{say}

Let $\mathbb{P}^n = \mathbb{P}(V^{n+1})$.
A twisted differential $1$-form $\omega\in \mathbb{P}(H^0(\mathbb{P}^n,\Omega^1_{\mathbb{P}^n}(2)))$ can be written as 
$$\omega = \sum_{0\leq i<j\leq n}a_{ij}(z_idz_j-z_jdz_i).$$
The matrix $M_{\omega} = (a_{ij})$ is skew-symmetric  of size $(n+1)$.
This gives rise to an isomorphism
\stepcounter{thm}
\begin{equation}\label{iso1}
\begin{array}{cccc}
\psi: & \mathbb{P}(H^0(\mathbb{P}^n,\Omega^1_{\mathbb{P}^n}(2))) & \longrightarrow & \mathbb{P}(\bigwedge^2V^{n+1}).\\
 & \omega & \longmapsto & M_{\omega}
\end{array}
\end{equation}

\begin{Lemma}\label{l1}
Let $\omega\in \mathbb{P}(H^0(\mathbb{P}^n,\Omega^1_{\mathbb{P}^n}(2)))$ be a twisted differential $1$-form, and 
$M_{\omega}$ the corresponding skew-symmetric matrix. 
Let $M_{i_0,...,i_{2k+3}}$ be the sub-Pfaffian of $M_{\omega}$ obtained by deleting the rows and the columns indexed by $j\in\{0,...,n\}\setminus \{i_0,...,i_{2k+3}\}$. Then we have  the following formula for $\omega\wedge (d\omega)^{k+1}$:
\stepcounter{thm}
\begin{equation}\label{eq1}
2^{k+1}\sum_{0\leq i_0<i_1<...<i_{2k+3}\leq n}M_{i_0,...,i_{2k+3}}\left(\sum_{j=0}^{2k+3}(-1)^{j}z_{i_j}(dz_{i_0}\wedge ... \wedge\widehat{dz_{i_j}}\wedge ...\wedge dz_{i_{2k+3}})\right).
\end{equation}
\end{Lemma}

\begin{proof}
We prove (\ref{eq1}) by induction on $k$. For $k=-1$ we have
$$\omega = \sum_{0\leq i_0<i_1\leq n}M_{i_0,i_{1}}(z_{i_0}dz_{i_1}-z_{i_1}dz_{i_0}) = \sum_{0\leq i<j\leq n}a_{ij}(z_idz_j-z_jdz_i).$$
By the induction hypothesis, $\omega\wedge (d\omega)^{k+1} = (\omega\wedge (d\omega)^k)\wedge d\omega$ is equal to
$$
\begin{array}{l}
2^{k}\sum_{0\leq i_0<i_1<...<i_{2k+1}\leq n}M_{i_0,...,i_{2k+1}}\left(\sum_{j=0}^{2k+1}(-1)^{j}z_{i_j}(dz_{i_0}\wedge ... \wedge\widehat{dz_{i_j}}\wedge ...\wedge dz_{i_{2k+1}})\right)\wedge \\ 
2\sum_{0\leq i<j\leq n}a_{ij}(dz_i\wedge dz_j)=\\
2^{k+1}\sum_{0\leq i_0<i_1<...<i_{2k+1}\leq n}M_{i_0,...,i_{2k+1}}\sum_{j_0,j_1\notin\{i_1,...,i_{2k+1}\}}(-1)^{i_0+...+i_{2k+1}+j_0+j_1}a_{j_0j_1}\\
\left(\sum_{j=0}^{2k+3}(-1)^{j}z_{i_j}(dz_{i_0}\wedge ... \wedge\widehat{dz_{i_j}}\wedge ...\wedge dz_{i_{2k+3}})\right)=\\
2^{k+1}\sum_{0\leq i_0<i_1<...<i_{2k+3}\leq n}M_{i_0,...,i_{2k+3}}\left(\sum_{j=0}^{2k+3}(-1)^{j}z_{i_j}(dz_{i_0}\wedge ... \wedge\widehat{dz_{i_j}}\wedge ...\wedge dz_{i_{2k+3}})\right),
\end{array} 
$$
which is exactly the formula in the statement.
\end{proof}

\begin{thm}\label{mainstrat}
Let $D_k\subseteq\mathbb{P}(H^0(\mathbb{P}^n,\Omega^1_{\mathbb{P}^n}(2)))$ be the variety parametrizing codimension one distributions on 
$\mathbb{P}^n = \mathbb{P}(V^{n+1})$ of class $\leq k$ and index $n-1$. Then $D_k = \sec_{k+1}(\mathbb{G}(1,n))$ and, 
via the identification induced by $\psi$, the stratification 
$$D_0\subseteq D_1\subseteq ...\subseteq D_{k-1}\subseteq ...\subseteq \mathbb{P}(H^0(\mathbb{P}^n,\Omega^1_{\mathbb{P}^n}(2)))$$
corresponds to the natural stratification
$$\mathbb{G}(1,n)\subseteq \sec_2(\mathbb{G}(1,n))\subseteq ...\subseteq \sec_k(\mathbb{G}(1,n))\subseteq ...\subseteq \mathbb{P}(\bigwedge^2V^{n+1}).$$
\end{thm}
\begin{proof}
Consider the isomorphism in \eqref{iso1}.
By Formula (\ref{eq1}), via the identification induced by $\psi$, the variety $D_k$ is defined by the vanishing of all the sub-Pfaffians $M_{i_0,...,i_{2k+3}}$ of size $2k+4$ of $M_{\omega}$. 
On the other hand,  these sub-Pfaffians are known to generate the ideal of $\sec_{k+1}(\mathbb{G}(1,n))$
(see for instance \cite[Section 10]{LO}). 
\end{proof}

%
%

\section{Weighted complete intersections} \label{section:wci}
In this section we address Fano distributions on weighted complete intersections by cohomological computations. 

\begin{say}[Fano manifolds of Picard number one]\label{restriction_exact_seqs} 
Let $Y$ be an $n$-dimensional Fano manifold  with $\rho(Y)=1$, and denote by $\sO_Y(1)$ the ample generator of $\Pic(Y)$. Let $X\in \big|\sO_Y(d) \big|$ be a smooth divisor. We have the following exact sequences:
\stepcounter{thm}
\begin{equation}\label{restriction2}
0 \ \to \ \Omega_Y^{q}(t-d) \ \to \ \Omega_Y^{q}(t) \ \to \ \Omega_Y^{q}(t)|_X \ \to \ 0,
\end{equation}
and
\stepcounter{thm}
\begin{equation}\label{restriction1}
0 \ \to \ \Omega_X^{q-1}(t-d) \ \to \ \Omega_Y^{q}(t)|_X \ \to \ \Omega_X^{q}(t) \ \to \ 0.  
\end{equation}
By taking cohomology in (\ref{restriction2}) for $t = -d$, (\ref{restriction1}) for $t = -d$, (\ref{restriction2}) for $t = 0$, and (\ref{restriction1}) for $t = 0$

 \[
  \begin{tikzpicture}[xscale=4.5,yscale=-1.2]
    \node (A0_0) at (0, 0) {$...\rightarrow H^{p-1}(Y,\Omega_Y^{q-1}(-d))$};
    \node (A0_1) at (1, 0) {$H^{p-1}(Y,\Omega_Y^{q-1})$};
    \node (A0_2) at (2, 0) {$H^{p-1}(X,\Omega_{Y|X}^{q-1})\rightarrow ...$};
    \node (A1_0) at (0, 1) {$...\rightarrow  H^{p-1}(X,\Omega_X^{q-2}(-d))$};
    \node (A1_1) at (1, 1) {$H^{p-1}(X,\Omega_{Y|X}^{q-1})$};
    \node (A1_2) at (2, 1) {$H^{p-1}(X,\Omega_X^{q-1})\rightarrow ...$};
    \node (A2_0) at (0, 2) {$...\rightarrow  H^{p-1}(X,\Omega_X^{q-1})$};
    \node (A2_1) at (1, 2) {$H^{p-1}(X,\Omega_{Y}^{q}(d)|_X)$};
    \node (A2_2) at (2, 2) {$H^{p-1}(X,\Omega_X^{q}(d))\rightarrow ...$};
    \node (A3_0) at (0, 3) {$...\rightarrow  H^{p-1}(Y,\Omega_Y^{q}(d))$};
    \node (A3_1) at (1, 3) {$H^{p-1}(X,\Omega_{Y}^{q}(d)|_X)$};
    \node (A3_2) at (2, 3) {$H^{p}(Y,\Omega_Y^{q})\rightarrow ...$};
    \path (A2_1) edge [-,double distance=1.5pt]node [auto] {$\scriptstyle{}$} (A3_1);
    \path (A1_2) edge [-,double distance=1.5pt]node [auto] {$\scriptstyle{}$} (A2_0);
    \path (A2_1) edge [->]node [auto] {$\scriptstyle{}$} (A2_2);
    \path (A0_2) edge [-,double distance=1.5pt]node [auto] {$\scriptstyle{}$} (A1_1);
    \path (A0_0) edge [->]node [auto] {$\scriptstyle{}$} (A0_1);
    \path (A1_0) edge [->]node [auto] {$\scriptstyle{}$} (A1_1);
    \path (A3_0) edge [->]node [auto] {$\scriptstyle{}$} (A3_1);
    \path (A0_1) edge [->]node [auto] {$\scriptstyle{}$} (A0_2);
    \path (A1_1) edge [->]node [auto] {$\scriptstyle{}$} (A1_2);
    \path (A2_0) edge [->]node [auto] {$\scriptstyle{}$} (A2_1);
    \path (A3_1) edge [->]node [auto] {$\scriptstyle{}$} (A3_2);
  \end{tikzpicture}
  \]
we get the map
\stepcounter{thm}
\begin{equation}\label{alfa}
\alpha:H^{p-1}(Y,\Omega^{q-1}_Y)\rightarrow H^p(Y,\Omega_Y^q),
\end{equation}
and the map
\stepcounter{thm}
\begin{equation}\label{beta}
\beta:H^{p-1}(X,\Omega^{q-1}_X)\rightarrow H^{p-1}(X,\Omega_{Y}^q(d)|_X).
\end{equation}

\begin{Lemma}\label{PW}
If $h^{p-1}(Y,\Omega_{Y}^{q-1}) \neq 0$ and $p+q\leq n$, the the map $\beta$ in (\ref{beta}) is non-zero.
\end{Lemma}
\begin{proof}
By \cite[Lemma 1.2]{PW} the map $\alpha$ in (\ref{alfa}) is the cup product with $c_1(\mathcal{O}_Y(d))$. Therefore, by the weak Lefschetz theorem $\alpha$ is injective if $p+q\leq n$, and since $h^{p-1}(Y,\Omega_{Y}^{q-1}) \neq 0$ the map $\beta$ is non-zero.
\end{proof}
\end{say}

\begin{say}[Cohomology of $\overline{\Omega}^{q}_{\p}(t)$]\label{WPS}
Let $\p=\p(a_0,\ldots, a_N)=\Proj \big(S(a_0,\ldots, a_N)\big)$ be as in Paragraph~\ref{WPS0}.

Consider the sheaves of $\sO_{\p}$-modules $\overline{\Omega}^{q}_{\p}(t)$ defined in \cite[Section 2.1.5]{dolgachev} for $q, t\in\mathbb{Z}$, $q\ge 0$. 
If $U\subset \p$ denotes the smooth locus of $\p$, and $\sO_{U}(t)$ is the line bundle obtained by restricting 
$\sO_{\p}(t)$ to $U$, then ${\overline{\Omega}^{q}_{\p}(t)}_{|U}=\Omega^q_{U}\otimes \sO_{U}(t)$. 
The cohomology groups  $H^p\big(\p, \overline{\Omega}^{q}_{\p}(t)\big)$ are described in \cite[Section 2.3.2]{dolgachev}:
\begin{itemize}
	\item[-] $h^0\big(\p, \overline{\Omega}^{q}_{\p}(t)\big)=\sum_{i=0}^q \Big((-1)^{i+q} \sum_{\#J=i}\dim_{\mathbb{C}}\big(S_{t-a_J}\big)\Big)$, where $J\subset \{0, \ldots, N\}$ and $a_J:=\sum_{i\in J}a_i$; 
	\item[-] $h^0\big(\p, \overline{\Omega}^{q}_{\p}(t)\big) = 0$ if $t <\min \{\sum_{j\in J}a_{i_j}\: | \: \#J = q\}$;
	\item[-] $h^p\big(\p, \overline{\Omega}^{q}_{\p}(t)\big)=0$ if $p\not\in \{0, q,N\}$.
	\item[-] $h^p\big(\p, \overline{\Omega}^{p}_{\p}(t)\big)=0$ if $t\neq 0$ and $p\notin\{0,N\}$.
\end{itemize}
In particular, if $q\geq 1$, then 
\stepcounter{thm}
\begin{equation}\label{lcwps}
h^0(\mathbb{P},\Omega^q_{\mathbb{P}}(t)) = 0 \ \text{ for any } \ t\leq q.
\end{equation}
\end{say}

When $\p(a_0,\ldots, a_N) = \mathbb{P}^N$ is a projective space we have the classical Bott's formulas.
\begin{say}[Bott's formulas]\label{bott}
Let $p,q$ and $t$ be integers, with $p$ and $q$ 
non-negative.  Then
\begin{equation*}
h^p\big(\p^N,\Omega_{\p^N}^q(t)\big) =
\begin{cases}
\binom{t+N-q}{t}\binom{t-1}{q} & \text{ for } p=0, 0\le q\le N \text{ and } t>q,\\
1 & \text{ for } t=0 \text{ and } 0\le p=q\le N,\\
\binom{-t+q}{-t}\binom{-t-1}{N-q} & \text{ for } p=N, 0\le q\le N \text{ and } t<q-N,\\
0 & \text{otherwise.}
\end{cases}
\end{equation*}
\end{say}

Now assume that $\p$ has only isolated singularities, let $d>0$ be such that $\sO_{\p}(d)$ is a line bundle generated by global sections, and $X\in \big|\sO_{\p}(d)\big|$ a smooth hypersurface. We will use the cohomology groups $H^p\big(\p, \overline{\Omega}^{q}_{\p}(t)\big)$ to compute some cohomology groups 
$H^p\big(X,\Omega_{X}^q(t)\big)$. Note that $X$ is contained in the smooth locus of $\p$, so we have an exact sequence as in \eqref{restriction1}:
\stepcounter{thm}
\begin{equation}\label{restriction1_P}
0 \to \ \Omega_X^{q-1}(t-d) \ \to \ \overline{\Omega}^{q}_{\p}(t)|_X \ \to \ \Omega_X^{q}(t) \to  0.
\end{equation}
Tensoring the sequence 
$$
0 \to\ \sO_{\p}(-d)  \ \to \ \sO_{\p}  \ \to \ \sO_{X} \ \to 0.
$$
with the sheaf $\overline{\Omega}^{q}_{\p}(t)$, and noting that $\overline{\Omega}^{q}_{\p}(t)\otimes \sO_{\p}(-d) \cong \overline{\Omega}^{q}_{\p}(t-d)$, we get an exact sequence  as in \eqref{restriction2}:
\stepcounter{thm}
\begin{equation}\label{restriction2_P}
0 \to \overline{\Omega}^{q}_{\p}(t-d) \ \to \ \overline{\Omega}^{q}_{\p}(t) \ \to \ \overline{\Omega}^{q}_{\p}(t)|_X \to 0.
\end{equation}

\begin{say}[Weighted complete intersections]\label{wci}
Let $X\subset \mathbb{P}(a_0,...,a_N)$ be a smooth $n$-dimensional weighted complete intersection in a weighted projective space. Then $X$ is the scheme-theoretic zero locus of $c = N-n$ weighted homogeneous polynomials $f_1,...,f_c$ of degrees $d_1,...,d_c$. By \cite[Theorem 3.2.4]{dolgachev}, $\Pic(X)\cong\mathbb{Z}$ if 
$n\geq 3$. Furthermore, by \cite[Theorem 3.3.4]{dolgachev},
\stepcounter{thm}
\begin{equation}\label{canwci}
\omega_X\cong\mathcal{O}_X\left(\sum_{j=1}^cd_j-\sum_{i=0}^{N}a_i\right)
\end{equation}
In particular, when $X$ is Fano, its index is  $\iota_X := \sum_{i=0}^{N}a_i-\sum_{j=1}^cd_j$. 

Let $S_t$ be the $t$-th graded part of $S/(f_1,...,f_c)$. By \cite[Lemma 7.1]{CR},
\stepcounter{thm}
\begin{equation}\label{owci}
H^i(X,\mathcal{O}_X(t))\cong
\left\lbrace\begin{array}{lll}
S_t & \rm{if} & i=0; \\ 
0 & \rm{if} & 1\leq i\leq n-1;\\ 
S_{-t+\iota_X} & \rm{if} & i  = n.
\end{array} \right.
\end{equation} 
\end{say}

\begin{say}\label{cwci}
Finally, by \cite[Satz 8.11]{flenner81} we have the following formulas for the cohomology of $X$:
\begin{enumerate}
	\item[-] $h^q(X,\Omega_{X}^q) =1 $ for $0\le q\le n$, $q\neq \frac{n}{2}$.
	\item[-] $h^p\big(X,\Omega_{X}^q(t)\big) = 0$ in the following cases
		\begin{itemize}
			\item[-] $0<p<n$, $p+q\neq n$ and either $p\neq q$ or $t\neq 0$;
			\item[-] $p+q > n$ and $t>q-p$;
			\item[-] $p+q < n$ and $t<q-p$.
		\end{itemize} 
\end{enumerate}
\end{say}

For the rest of this section we work under the following assumptions.

\begin{Assumption}\label{asswci}
Let $ \mathbb{P}:=\mathbb{P}(a_0,...,a_N)$ be a weighted projective space with at most isolated singularities.
Let $X\subset \mathbb{P}$ be a weighted complete intersection of dimension $n\geq 3$,
defined by weighted homogeneous polynomials $f_1,...,f_c$ of degrees $d_1,...,d_c$, with $2 \leq d_1\leq d_2\leq ... \leq d_c$.
We assume that the weighted complete intersection $X_i = \{f_1= ... = f_i = 0\}$ is smooth for every $i\in \{1,\cdots, c\}$. 
\end{Assumption}

\begin{Lemma}\label{lf1}
Under Assumptions \ref{asswci}, we have:
$$
H^0(X,\Omega_X^{q-1}(t))=0
$$
for any $2\leq q\leq n$ and $t\leq q-1$.
\end{Lemma}

\begin{proof}
If $t< q-1$, then the result follows from (\ref{cwci}). 
So it is enough to consider the case $t = q-1$. 
Set $X_0 = \mathbb{P}$, and  $X_i = \{f_1= ... = f_i = 0\}$. 
Then $X_i$ is a divisor in $X_{i-1}$ cut out by a homogeneous polynomial of degree $d_i$ for any $i = 1,...,c$. 
By \eqref{lcwps}, $H^0(\mathbb{P},\Omega^{q-1}_{\mathbb{P}}(q-1))=0$. 
We proceed by induction on $\codim(X_i)$. By (\ref{restriction2}) and (\ref{restriction1}) for $X_i\in |\mathcal{O}_{X_{i-1}}(d_i)|$ we have the following exact sequences:
\stepcounter{thm}
\begin{equation}\label{lemmaeq1}
0  \to \Omega_{X_{i-1}}^{q-1}(q-1-d_i)  \to  \Omega_{X_{i-1}}^{q-1}(q-1)  \to  \Omega_{X_{i-1}}^{q-1}(q-1)_{|X_i}  \to  0 ,
\end{equation} 
\stepcounter{thm}
\begin{equation}\label{lemmaeq2}
0  \to  \Omega_{X_{i}}^{q-2}(q-1-d_i)  \to  \Omega_{X_{i-1}}^{q-1}(q-1)_{|X_i}
  \to  \Omega_{X_i}^{q-1}(q-1) \to 0 .
\end{equation}
By the induction hypothesis, $H^0(X_{i-1},\Omega_{X_{i-1}}^{q-1}(q-1))=0$. 
Note that $q<\dim(X_{i-1})$. 
Moreover, $q-1 = 1$ if and only if $q = 2$, which implies that $q-1-d_i\neq 0$. Therefore, (\ref{cwci}) yields $H^1(X_{i-1},\Omega_{X_{i-1}}^{q-1}(q-1-d_i))=0$. 
Hence 
\stepcounter{thm}
\begin{equation}\label{ef1}
H^0(X_{i-1},\Omega_{X_{i-1}}^{q-1}(q-1)_{|X_i}) = 0.
\end{equation}
By (\ref{cwci}), $H^1(X_i,\Omega_{X_{i}}^{q-2}(q-1-d_i))=0$ for any $(q,d_i)\neq (3,2)$. 
Therefore, if either $q\neq 3$ or $d_i > 2$, we conclude that $H^0(X_{i},\Omega_{X_{i}}^{q-1}(q-1)) = 0$ by (\ref{ef1}). 
Let us assume $(q,d_i)=(3,2)$. Then we have maps
\stepcounter{thm}
\begin{equation}\label{lemmaeq4}
0\rightarrow H^0(X_i,\Omega_{X_i}^2(2))\rightarrow H^1(X_i,\Omega_{X_i}^1)\xrightarrow{\beta} H^1(X_{i},\Omega_{X_{i-1}}^2(2)_{|X_i})\rightarrow H^1(X_i,\Omega^2_{X_i}(2)).
\end{equation}
The map $\beta$ is exactly the map in (\ref{beta}) for $p = q = 2$ and $d = 2$. 
Since $\dim(X_i)\geq 3$, we have $H^1(X_{i-1}),\Omega^1_{X_{i-1}})\cong H^1(X_i,\Omega_{X_i}^1)\cong \mathbb{C}$ by (\ref{cwci}). 
By Lemma \ref{PW}, $\beta$ is non-zero, and since $h^1(X_i,\Omega_{X_i}^1)=1$ this yields $H^0(X_i,\Omega_{X_i}^2(2)) = \Ker(\beta) = 0$.
\end{proof}

\begin{Proposition}\label{f1}
Under Assumptions \ref{asswci}, for any $t\in\mathbb{Z}$ and $q\leq n$ there exists a natural map
$$
r_q:H^0(\mathbb{P},\Omega_{\mathbb{P}}^{q}(t))\rightarrow H^0(X,\Omega_{X}^{q}(t))
$$
If $t\leq d_1$, then $r_1$ is an isomorphism. If $t \leq d_1+q-1$, then $r_q$ is injective for any $q\geq 1$.
\end{Proposition}

\begin{proof}
By (\ref{restriction2}) and (\ref{restriction1}), for $X_i\in |\mathcal{O}_{X_{i-1}}(d_i)|$ we have the following exact sequences:
$$
\begin{array}{c}
0 \ \to \ \Omega_{X_{i-1}}^{q}(t-d_i) \ \to \ \Omega_{X_{i-1}}^{q}(t) \ \to \ \Omega_{X_{i-1}}^{q}(t)_{|X_i} \ \to \ 0 ,\\ 
0 \ \to \ \Omega_{X_{i}}^{q-1}(t-d_i) \ \to \ \Omega_{X_{i-1}}^{q}(t)_{|X_i}
 \ \to \ \Omega_{X_i}^{q}(t) \ \to \ 0 .
\end{array} 
$$
Taking cohomology we get:
    \[
  \begin{tikzpicture}[xscale=5.0,yscale=-1.2]
    \node (A0_0) at (0, 0) {$0\to H^0(X_{i-1},\Omega^{q}_{X_{i-1}}(t-d_i))$};
    \node (A0_1) at (1, 0) {$H^0(X_{i-1},\Omega^{q}_{X_{i-1}}(t))$};
    \node (A0_2) at (2, 0) {$H^0(X_{i},\Omega^{q}_{X_{i-1}}(t)_{|X_i})$};
    \node (A1_0) at (0, 1) {$0\to H^0(X_{i},\Omega^{q-1}_{X_{i}}(t-d_i))$};
    \node (A1_1) at (1, 1) {$H^0(X_{i},\Omega^{q}_{X_{i-1}}(t)_{|X_i})$};
    \node (A1_2) at (2, 1) {$H^0(X_{i},\Omega^{q}_{X_{i}}(t))$.};
    \path (A0_2) edge [-,double distance=1.5pt]node [auto] {$\scriptstyle{}$} (A1_1);
    \path (A0_0) edge [->]node [auto] {$\scriptstyle{}$} (A0_1);
    \path (A0_1) edge [->]node [auto] {$\scriptstyle{\alpha_{i}^q}$} (A0_2);
    \path (A1_0) edge [->]node [auto] {$\scriptstyle{}$} (A1_1);
    \path (A1_1) edge [->]node [auto] {$\scriptstyle{\beta_i^q}$} (A1_2);
  \end{tikzpicture}
  \]
We define 
$$
r_q := \beta^q_c\circ\alpha^q_c\circ\beta^q_{c-1}\circ\alpha^q_{c-1}\circ ... \circ \beta^q_1\circ\alpha^q_1 . 
$$

First we consider the case $q = 1$ and $t < d_1$.
For $i=1$,  the formulas in Section \ref{WPS} yield $H^i(\mathbb{P},\Omega_{\mathbb{P}}^1(t-d_1))=0$ for $i = 0,1$. 
So $\alpha_1^{1}$ is an isomorphism. By (\ref{owci}), we have 
$$
H^0(X_i,\mathcal{O}_{X_i}(t-d_i))=H^1(X_i,\mathcal{O}_{X_i}(t-d_i)) = 0
$$ 
for any $i = 1,...,c$. So $\beta^1_i$ is an isomorphism for $i = 1,...,c$. 
Since $n\geq 3$, the formulas in (\ref{cwci}) yield $H^0(X_{i-1},\Omega^{1}_{X_{i-1}}(t-d_i)) = H^1(X_{i-1},\Omega^{1}_{X_{i-1}}(t-d_i))=0$.
So is $\alpha^1_i$ is an isomorphism for  $i=2,...,c$. We conclude that $r_1$ is an isomorphism.

Now suppose that  $q = 1$ and $t = d_1$. 
By (\ref{cwci}), we still have $H^0(X_{i-1},\Omega_{X_{i-1}}^1)=0$. On the other hand, again by (\ref{cwci}) we have that $h^1(X_{i-1},\Omega_{X_{i-1}}^1)=1$. Let us consider the following diagram:
  \[
  \begin{tikzpicture}[xscale=4.5,yscale=-1.2]
    \node (A0_0) at (0, 0) {$0\to H^{0}(X_{i-1},\Omega^1_{X_{i-1}}(t))$};
    \node (A0_1) at (1, 0) {$H^{0}(X_i,\Omega^1_{X_{i-1}}(t)_{|X_{i}})$};
    \node (A0_2) at (2, 0) {$H^{1}(X_{i-1},\Omega^1_{X_{i-1}})\cong\mathbb{C}\to 0$};
    \node (A1_0) at (0, 1) {$0\to H^0(X_i,\mathcal{O}_{X_i})\cong\mathbb{C}$};
    \node (A1_1) at (1, 1) {$H^{0}(X_i,\Omega^1_{X_{i-1}}(t)_{|X_i})$};
    \node (A1_2) at (2, 1) {$H^{0}(X_i,\Omega^1_{X_i}(t))\to 0$.};
    \path (A0_0) edge [->]node [auto] {$\scriptstyle{\alpha_i}$} (A0_1);
    \path (A0_1) edge [-,double distance=1.5pt]node [auto] {$\scriptstyle{}$} (A1_1);
    \path (A1_0) edge [->]node [auto] {$\scriptstyle{\gamma_i}$} (A1_1);
    \path (A0_1) edge [->]node [auto] {$\scriptstyle{\beta_i}$} (A0_2);
    \path (A1_1) edge [->]node [auto] {$\scriptstyle{\delta_i}$} (A1_2);
  \end{tikzpicture}
  \]
We already have an isomorphism $H^0(X_{i-1},\Omega^1_{X_{i-1}}(t))\cong H^{0}(\mathbb{P},\Omega^1_{\mathbb{P}}(t))$. 
Note that $df_i\notin \im(\alpha_i)$ and $H^1(X_{i-1},\Omega_{X_{i-1}}^1)$ is generated by $\beta_i(df_i)$. On the other hand $df_i$ generates $\Ker(\delta_i)\cong H^0(X_i,\mathcal{O}_{X_i})$. Therefore
$$
\phi_i := \delta_i\circ\alpha_i: H^{0}(X_{i-1},\Omega^1_{X_{i-1}}(t))\rightarrow H^{0}(X_i,\Omega^1_{X_i}(t))
$$
is an isomorphism for any $i = 1,...,c$. 

Finally, suppose that $q\geq 2$. 
Since $t -d_1\leq q-1$, Lemma \ref{lcwps} yields $H^0(\mathbb{P},\Omega_{\mathbb{P}}^q(t-d_1))=0$ and $\alpha^q_1$ is injective. Furthermore, since $t -d_i\leq q-1$ by Lemma \ref{lf1} we get $H^0(X_{i-1},\Omega^{q}_{X_{i-1}}(t-d_i)) = 0$ for $i\geq 2$, and $H^0(X_{i},\Omega^{q-1}_{X_{i}}(t-d_i))=0$ for $i\geq 1$. 
So $r_q$ is injective.
\end{proof}

\begin{Corollary}\label{corwci}
Under Assumptions \ref{asswci}, there exists a natural isomorphism
$$r_1:H^0(\mathbb{P},\Omega_{\mathbb{P}}^{1}(2))\rightarrow H^0(X,\Omega_{X}^{1}(2))$$
preserving the class $k$  of differential forms when $k \leq \lfloor\frac{\dim(X)-1}{2}\rfloor$, and mapping 
forms of class $k \geq \lfloor\frac{\dim(X)-1}{2}\rfloor$ to forms of maximal class in $X$.
\end{Corollary}

\begin{proof}
By Proposition \ref{f1}, there is an isomorphism 
$$
r_1:H^0(\mathbb{P},\Omega_{\mathbb{P}}^{1}(2))\rightarrow H^0(X,\Omega_{X}^{1}(2)).
$$
Let $\omega\in H^0(\mathbb{P},\Omega_{\mathbb{P}}^{1}(2))$ be a form of class $k$. 
Then $0\neq \omega\wedge(d\omega)^k\in H^0(\mathbb{P},\Omega_{\mathbb{P}}^{2k+1}(2k+2))$. 

Suppose first that  $k \leq \lfloor\frac{\dim(X)-1}{2}\rfloor$.
By Proposition \ref{f1}, the map 
$$r_{2k+1}:H^0(\mathbb{P},\Omega_{\mathbb{P}}^{2k+1}(2k+2))\rightarrow H^0(X,\Omega_{X}^{2k+1}(2k+2))$$
is injective. Therefore 
$r_1(\omega)\wedge\big(dr_1(\omega)\big)^k= r_{2k+1}(\omega\wedge(d\omega)^k)\neq 0$, and
$r_1(\omega)$ also has  class $k$. 

Suppose now that $k \geq \lfloor\frac{\dim(X)-1}{2}\rfloor$, and set $h:=\lfloor\frac{\dim(X)-1}{2}\rfloor$. 
By Proposition \ref{f1}, $\omega\wedge(d\omega)^h\neq 0$ in $\mathbb{P}$ implies that 
$r_1(\omega)\wedge\big(dr_1(\omega)\big)^h= r_{2h+1}(\omega\wedge(d\omega)^h)\neq 0$ in $X$. 
Hence $\omega$ restricts to a form of maximal class $h$ in $X$.
\end{proof}

\begin{thm}\label{main-ci-wps}
Let $X\subset \mathbb{P}=\p(\underbrace{1, \ldots, 1}_{\ell + 1 \text{ times }},a_{\ell+1},\ldots, a_N)$,  $1 <  a_{\ell+1} < \cdots < a_N$,
be a complete intersection as in Assumptions \ref{asswci}. 
Then
\begin{enumerate}

\item $H^0(X,\Omega^1_{X}(1))=0$.

\item Let $\varphi:\p\dasharrow \p^\ell$ be the rational map defined by $(z_0:\cdots: z_\ell)$,
$$
D_k\subseteq\mathbb{P}(H^0(\mathbb{P}^\ell,\Omega^1_{\mathbb{P}^\ell}(2)))
$$ 
the subvariety parametrizing  distributions of class $\leq k$ on $\mathbb{P}^\ell$, and 
$$
\overline{D}_k\subseteq\mathbb{P}(H^0(X,\Omega^1_{X}(2)))
$$ 
the subset parametrizing distributions of class $\leq k$ on $X$. 

Then $\overline{D}_{\lfloor\frac{\ell-1}{2}\rfloor} = H^0(X,\Omega^1_{X}(2))$, and $\varphi$ induces an isomorphism 
$$\varphi^*: H^0(\mathbb{P}^\ell,\Omega_{\mathbb{P}^\ell}^{1}(2))\cong H^0(X,\Omega_{X}^{1}(2))$$
that maps $D_k$ isomorphically onto $\overline{D}_k$ for any $k < \min\left\{\lfloor\frac{\ell-1}{2}\rfloor, \lfloor\frac{\dim(X)-1}{2}\rfloor\right\}$. 
\end{enumerate}
\end{thm}

\begin{proof}
The first statement follows from Lemma \ref{lf1}. Furthermore, by Corollary \ref{corwci}, the restriction map $r_1:H^0(\mathbb{P},\Omega_{\mathbb{P}}^{1}(2))\rightarrow H^0(X,\Omega_{X}^{1}(2))$
is an isomorphism that preserves the class $k$  of differential forms when $k \leq \lfloor\frac{\dim(X)-1}{2}\rfloor$, and maps
forms of class $k \geq \lfloor\frac{\dim(X)-1}{2}\rfloor$ on $\p$ to forms of maximal class on $X$.
The second statement then follows from Proposition~\ref{weighted fol}.
\end{proof}

Finally, we prove some facts on infinitesimal deformations of weighted complete intersections coming as a byproduct of the cohomological results in this section.
 
Recall that the tangent and obstruction spaces to deformations of a smooth variety $X$ are given by $H^1(X,T_X)$ and $H^2(X,T_X)$, respectively. Furthermore, the tangent space to $\Aut(X)$ at the identity is given by $T_{Id}\Aut(X)\cong H^0(X,T_X)$ (see for instance  \cite[Chapter 1]{Se}). 

\begin{Proposition}\label{autwci}
Under Assumptions \ref{asswci}, if 
$$
\sum_{i=0}^{N}a_i-\sum_{j=1}^{c}d_j\leq n-1,
$$
then $\Aut(X)$ is finite. Furthermore, the first order infinitesimal deformations of $X$ are unobstructed of dimension 
$$
h^1(X,T_X) = \sum_{j=1}^{c}h^0(X,\mathcal{O}_X(d_j))-\dim \Aut(\mathbb{P}(a_0,...,a_N)).
$$
\end{Proposition}

\begin{proof}
By \eqref{canwci}, we have $T_X\cong\Omega_X^{n-1}\left(\sum_{i=0}^{N}a_i-\sum_{j=1}^{c}d_j\right)$. 
Lemma \ref{lf1} yields that $H^0(X,T_X) = 0$ provided that  $\sum_{i=0}^{N}a_i-\sum_{j=1}^{c}d_j\leq n-1$. 
By taking  cohomology of the exact sequence 
$$
0\to T_X\rightarrow T_{\mathbb{P}|X}\rightarrow \bigoplus_{j=1}^{c}\mathcal{O}_X(d_j)\to 0,
$$
we get the formula for $h^1(X,T_X)$ and the vanishing of $H^2(X,T_X)$.
\end{proof}

As an immediate consequence of Proposition \ref{autwci}, we recover the following well know result about complete intersections 
in projective spaces (see for instance \cite[Theorem 3.1]{Be}).

\begin{Corollary}
Let $X\subset \mathbb{P}^N$ be a smooth complete intersection as in Assumptions~\ref{asswci}. 
If $X$ is not a quadric hypersurface, then $\Aut(X)$ is finite. 
The first order infinitesimal deformations of $X$ are unobstructed of dimension 
$$
h^1(X,T_X) = \sum_{j=1}^{c}h^0(X,\mathcal{O}_X(d_j))- (N+1)^2+1.
$$
\end{Corollary}

%
%

\section{Grassmannians of lines and their linear sections}\label{section:G}
Let $\mathbb{G}(k,n) = Gr(k+1,V)$ be the Grassmannian of $k$-planes in $\mathbb{P}^n = \mathbb{P}(V)$. 
Recall that  $\mathbb{G}(k,n)$ carries two canonical homogeneous bundles: the universal bundle $\mathcal{S}$, of rank $k+1$,
and the universal quotient bundle $\mathcal{Q}$, rank  $n-k$. They  fit in the exact sequence
$$
0\to \mathcal{S}\rightarrow V\otimes\mathcal{O}_{\mathbb{G}(k,n)}\rightarrow\mathcal{Q}\to 0.
$$
Furthermore, $\bigwedge^{k+1}\mathcal{S}^{\vee}\cong\bigwedge^{n-k}\mathcal{Q}\cong\mathcal{O}_{\mathbb{G}(k,n)}(1)$ is the line bundle on $\mathbb{G}(k,n)$ inducing the Pl\"ucker embedding $\mathbb{G}(k,n)\hookrightarrow\mathbb{P}^N$, with $N = \binom{n+1}{k+1}-1$.

\begin{Lemma}\label{Grass1n}
Let $\mathbb{G}(1,n)\subset\mathbb{P}^N$ be the Grassmannian of lines in $\mathbb{P}^n$ embedded via the Pl\"ucker embedding, and let $X_i$ be a codimension $i$ smooth linear section of $\mathbb{G}(1,n)$. Then 
$$H^0(X_i,\Omega_{X_i}^1(1))=0$$
for any $0\leq i \leq 2(n-1)-3$.
\end{Lemma}

\begin{proof}
By \cite[Lemma 0.1]{PW}, $H^0(\mathbb{G}(1,n),\Omega^1_{\mathbb{G}(1,n)}(1))=0$. We proceed by induction on $i = \codim_{\mathbb{G}(1,n)}(X_i)$.
Since $\dim(X_i)\geq 3$,  $\Pic(X_i)\cong\mathbb{Z}$ by Lefschetz hyperplane theorem.
By Paragraph~\ref{restriction_exact_seqs}, there are exact sequences:
$$
0\to \Omega^1_{X_{i-1}}\rightarrow \Omega^1_{X_{i-1}}(1)\rightarrow\Omega^1_{X_{i-1}}(1)_{|X_i}\to 0,
$$
$$
0\to\mathcal{O}_{X_i}\rightarrow \Omega^1_{X_{i-1}}(1)_{|X_i}\rightarrow \Omega^1_{X_{i}}(1)\to 0.
$$
We have $h^{1,1}(\mathbb{G}(1,n)) = h^1(\mathbb{G}(1,n),\Omega^1_{\mathbb{G}(1,n)}) = 1$ (see for instance \cite[Section 1.4]{Xu}). 
By the weak Lefschetz theorem, $h^{1,1}(X_i) = h^1(X_i,\Omega^1_{X_i}) = 1$ for $0\leq i \leq 2(n-1)-3$. 
By the induction hypothesis, $H^0(X_{i-1},\Omega_{X_{i-1}}^1(1))=0$. 
So taking cohomology of the above exact sequences we get:
\stepcounter{thm}
\begin{equation}\label{g1}
0\to H^0(X_{i-1},\Omega_{X_{i-1}}^1(1)_{|X_i})\rightarrow H^1(X_{i-1},\Omega^1_{X_{i-1}})\cong\mathbb{C},
\end{equation}
\stepcounter{thm}
\begin{equation}\label{g2}
0\to \mathbb{C}\rightarrow H^0(X_{i-1},\Omega^1_{X_{i-1}}(1)_{|X_i})\rightarrow H^0(X_i,\Omega_{X_i}^1(1))\to 0.
\end{equation}
The injective morphism (\ref{g1}) yields $h^0(X_{i-1},\Omega_{X_{i-1}}^1(1)_{|X_i})\leq 1$. On the other hand, \eqref{g2} forces $h^0(X_{i-1},\Omega_{X_{i-1}}^1(1)_{|X_i})= 1$. Therefore, $H^0(X_{i-1},\Omega^1_{X_{i-1}}(1)_{|X_i})\cong\mathbb{C}$, and again by (\ref{g2}) we 
conclude that $H^0(X_i,\Omega_{X_i}^1(1))=0$.       
\end{proof}

\begin{Lemma}\label{Grass1n2}
Let $\mathbb{G}(1,n)\subset\mathbb{P}^N$ be the Grassmannian of lines in $\mathbb{P}^n$ embedded via the Pl\"ucker embedding. Then the restriction morphism
$$
r:H^0(\mathbb{P}^N,\Omega_{\mathbb{P}^N}^1(2))\rightarrow H^0(\mathbb{G}(1,n),\Omega_{\mathbb{G}(1,n)}^1(2))
$$
is an isomorphism.
\end{Lemma}

\begin{proof}
Set $\mathbb{G}:= \mathbb{G}(1,n)$, and let us consider the two exact sequences:
\stepcounter{thm}
\begin{equation}\label{exg1}
0\to \mathcal{I}_{\mathbb{G}}\otimes\Omega^1_{\mathbb{P}^N}(2)\rightarrow\Omega^1_{\mathbb{P}^N}(2)\rightarrow\Omega^1_{\mathbb{P}^N}(2)|\mathbb{G}\to 0
\end{equation}
\stepcounter{thm}
\begin{equation}\label{exg2}
0\to N_{\mathbb{G}/\mathbb{P}^N}^{\vee}(2)\rightarrow\Omega^1_{\mathbb{P}^N}(2)_{|\mathbb{G}}\rightarrow\Omega^1_{\mathbb{G}}(2)\to 0
\end{equation}
where $N_{\mathbb{G}/\mathbb{P}^N}^{\vee} = \mathcal{I}_{\mathbb{G}}/\mathcal{I}_{\mathbb{G}}^{2}$ is the conormal bundle of 
$\mathbb{G}\subset\mathbb{P}^N$. 
By Proposition \ref{weighted fol}, a differential form $\omega\in H^0(\mathbb{P}^N,\Omega^1_{\mathbb{P}^N}(2))$ of class $k$ can be written 
in suitable coordinates as 
$$
\sum_{i=0}^k (z_{2i}dz_{2i+1}-z_{2i+1}dz_{2i}).
$$
In particular, the zero locus of $\omega$ is a linear subspace of $\mathbb{P}^N$.
Hence, since $\mathbb{G}\subset\mathbb{P}^N$ is non-degenerate, the map 
$$
\begin{array}{cccc}
\alpha: & H^0(\mathbb{P}^N,\Omega^1_{\mathbb{P}^N}(2))) & \longrightarrow & H^0(\mathbb{G},\Omega^1_{\mathbb{P}^N}(2)_{|\mathbb{G}}))\\
 & \omega & \longmapsto & \omega_{|\mathbb{G}}
\end{array}
$$
induced by (\ref{exg1}) is injective. 
Given a quadratic form $Q\in I(\mathbb{G})$ in the ideal of $\mathbb{G}\subset\mathbb{P}^N$,
the contraction of the differential form $dQ\in H^0(\mathbb{G},\Omega^1_{\mathbb{P}^N}(2)_{|\mathbb{G}})$ with  the radial vector field is $2Q$. 
Therefore $dQ\notin\im(\alpha)$, and $\alpha$ is not surjective.
The degree two part $I(\mathbb{G})_2$ of $I(\mathbb{G})$ is generated, as a vector space, by the Pfaffians of $4\times 4$ minors of a  
skew-symmetric $(n+1)\times (n+1)$ matrix (see for instance \cite[Section 10]{LO}). Therefore
\stepcounter{thm}
\begin{equation}\label{idgrass}
I(\mathbb{G})_2\cong \bigwedge^4 V.
\end{equation} 
Recall that $N_{\mathbb{G}/\mathbb{P}^N}\cong \bigwedge^2\mathcal{Q}\otimes\mathcal{O}_{\mathbb{G}}(1)$ (see for instance \cite[Section 5]{To}). 
Therefore
$$
N_{\mathbb{G}/\mathbb{P}^N}^{\vee}\cong \bigwedge^2\mathcal{Q}^{\vee}\otimes\mathcal{O}_{\mathbb{G}}(-1)\cong\bigwedge^{n-3}\mathcal{Q}\otimes\bigwedge^{n-1}\mathcal{Q}^{\vee}\otimes\mathcal{O}_{\mathbb{G}}(-1)\cong\bigwedge^{n-3}\mathcal{Q}\otimes\mathcal{O}_{\mathbb{G}}(-2)
$$
and 
\stepcounter{thm}
\begin{equation}\label{cnorgrass}
N_{\mathbb{G}/\mathbb{P}^N}^{\vee}(2)\cong\bigwedge^{n-3}\mathcal{Q}
\end{equation}
Taking cohomology in (\ref{exg2}) and considering (\ref{cnorgrass}), we get
$$
0\to H^0(\mathbb{G}, N_{\mathbb{G}/\mathbb{P}^N}^{\vee}(2))\cong\bigwedge^{n-3}V\rightarrow H^0(\mathbb{G},\Omega^1_{\mathbb{P}^N}(2)_{|\mathbb{G}})\xrightarrow{\beta} H^0(\mathbb{G},\Omega^1_{\mathbb{G}}(2))\rightarrow ...
$$
By (\ref{idgrass}), we have $m = \binom{n+1}{n-3}$ independent generators $dQ_1,...,dQ_m\in H^0(\mathbb{G}, N_{\mathbb{G}/\mathbb{P}^N}^{\vee}(2))$, where $Q_1,...,Q_m$ are quadric forms generating $I(\mathbb{G})_2$. 
Then $\Ker(\beta)\cap\im(\alpha) = \{0\}$ and 
$$
r := \beta\circ\alpha:H^0(\mathbb{P}^N,\Omega_{\mathbb{P}^N}^1(2))\rightarrow H^0(\mathbb{G},\Omega_{\mathbb{G}}^1(2))
$$
is injective. To conclude, note that  $h^0(\mathbb{G},\Omega_{\mathbb{G}}^1(2))= 3\binom{n+2}{4}$ by  \cite[Section 3.3]{Sn},  $N = \binom{n+1}{2}$,
and   $h^0(\mathbb{P}^N,\Omega_{\mathbb{P}^N}^1(2)) = \binom{N+1}{2} = 3\binom{n+2}{4}$ by Bott's formulas \eqref{bott}.
\end{proof}

\begin{Remark}
Analogues of Lemma \ref{Grass1n2} do not hold  for higher twisted holomorphic forms, nor for $\mathbb{G}(k,n)$ when $k\geq 2$. 
For instance, for the Grassmannian $\mathbb{G}(1,4)\subset\mathbb{P}^9$ we have $h^0(\mathbb{G}(1,4),\Omega_{\mathbb{G}(1,4)}^1(3)) = 280$, while $h^0(\mathbb{P}^9,\Omega_{\mathbb{P}^9}^1(3)) = 330$. For the Grassmannian $\mathbb{G}(2,5)\subset\mathbb{P}^{19}$, we have $h^0(\mathbb{G}(2,5),\Omega_{\mathbb{G}(2,5)}^1(2)) = 189$, while $h^0(\mathbb{P}^{19},\Omega_{\mathbb{P}^{19}}^1(2)) = 190$.
\end{Remark}

\begin{Lemma}\label{Grass1n3}
Let $\mathbb{G}(1,n)\subset\mathbb{P}^N$ be the Grassmannian of lines in $\mathbb{P}^n$ embedded via the Pl\"ucker embedding, and let $X_i$ be a codimension $i$ smooth linear section of $\mathbb{G}(1,n)$. Then the restriction map 
$$r:H^0(\mathbb{G}(1,n),\Omega_{\mathbb{G}(1,n)}^1(2))\rightarrow H^0(X_i,\Omega_{X_i}^1(2))$$
is surjective for any $0\leq i \leq 2(n-1)-4$. Furthermore, $\dim(\ker(r)) = i(N+1)-\frac{i(i+1)}{2}$.
\end{Lemma}
\begin{proof}
First we claim that 
\stepcounter{thm}
\begin{equation}\label{claim1}
H^1(X_i,\Omega_{X_i}^1(1))=0.
\end{equation}
By \cite[Corollary 1]{LeP}, $H^1(\mathbb{G}(1,n),\Omega^1_{\mathbb{G}(1,n)}(1))=0$. 
We proceed by induction on $i$. Consider the two exact sequences: 
$$0\to \Omega^1_{X_{i-1}}\rightarrow \Omega^1_{X_{i-1}}(1)\rightarrow\Omega^1_{X_{i-1}}(1)_{|X_i}\to 0,$$
$$0\to\mathcal{O}_{X_i}\rightarrow \Omega^1_{X_{i-1}}(1)_{|X_i}\rightarrow \Omega^1_{X_{i}}(1)\to 0.$$
Recall that all the non-trivial cohomology classes in $\mathbb{G}(1,n)$ are generated by algebraic cycles which are closures of affine spaces. Therefore, all the Hodge numbers $h^{i,j}(\mathbb{G}(1,n))$ are zero for $i\neq j$. Furthermore, the Lefschetz hyperplane theorem yields $h^{i,j}(X_i) = h^{i,j}(\mathbb{G}(1,n))$ when $i+j < \dim(X_i)$.
In our case,  $\dim(X_i)\geq 4$, and hence $h^2(X_i,\Omega^1_{X_i}) = h^{2,1}(\mathbb{G}(1,n)) = 0$. 
Furthermore, by the induction hypothesis, $H^1(X_{i-1},\Omega_{X_{i-1}}^1(1))=0$, and hence $H^1(X_{i},\Omega^1_{X_{i-1}}(1)_{|X_i})=0$ as well. 
To conclude the proof of \eqref{claim1}, note that $H^2(X_i,\mathcal{O}_{X_i})=0$.

We return to the restriction map 
$$r:H^0(\mathbb{G}(1,n),\Omega_{\mathbb{G}(1,n)}^1(2))\rightarrow H^0(X_i,\Omega_{X_i}^1(2)).$$
First  consider the case $i = 1$. By Paragraph~\ref{restriction_exact_seqs} there are exact sequences:
$$0\to \Omega^1_{\mathbb{G}(1,n)}(1)\rightarrow \Omega^1_{\mathbb{G}(1,n)}(2)\rightarrow\Omega^1_{\mathbb{G}(1,n)}(2)_{|X_1}\to 0,$$
$$0\to\mathcal{O}_{X_1}(1)\rightarrow \Omega^1_{\mathbb{G}(1,n)}(2)_{|X_1}\rightarrow \Omega^1_{X_{1}}(2)\to 0.$$
By Lemma \ref{Grass1n3} and \cite[Corollary 1]{LeP}, we have 
$$H^0(\mathbb{G}(1,n),\Omega^1_{\mathbb{G}(1,n)}(1)) = H^1(\mathbb{G}(1,n),\Omega^1_{\mathbb{G}(1,n)}(1))=0.$$
Furthermore, $H^0(X_1,\mathcal{O}_{X_1}(1))\cong \mathbb{C}^N$, and $H^1(X_1,\mathcal{O}_{X_1}(1)) = 0$. Therefore, the restriction map 
$$r_1:H^0(\mathbb{G}(1,n),\Omega_{\mathbb{G}(1,n)}^1(2))\rightarrow H^0(X_1,\Omega_{X_1}^1(2))$$
is surjective with kernel of dimension $N$. 

In general, consider the two exact sequences:
$$0\to \Omega^1_{X_{i-1}}(1)\rightarrow \Omega^1_{X_{i-1}}(2)\rightarrow\Omega^1_{X_{i-1}}(2)_{|X_i}\to 0,$$
$$0\to\mathcal{O}_{X_i}(1)\rightarrow \Omega^1_{X_{i-1}}(2)_{|X_i}\rightarrow \Omega^1_{X_{i}}(2)\to 0.$$
By Lemma \ref{Grass1n3} and  \eqref{claim1}, we have
$H^0(X_{i-1},\Omega^1_{X_{i-1}}(1)) =  H^1(X_{i-1},\Omega^1_{X_{i-1}}(1)) = 0$. Since $H^1(X_i,\mathcal{O}_{X_i}(1))=0$, the restriction map 
$$r_i:H^0(X_{i-1},\Omega_{X_{i-1}}^1(2))\rightarrow H^0(X_i,\Omega_{X_i}^1(2))$$
is surjective and its kernel has dimension $h^0(X_i,\mathcal{O}_{X_i}(1)) = N-i+1$. The composition map
$$r = r_i\circ r_{i-1}\circ ...\circ r_1:H^0(\mathbb{G}(1,n),\Omega_{\mathbb{G}(1,n)}^1(2))\rightarrow H^0(X_i,\Omega_{X_i}^1(2))$$
is surjective, and $\ker(r)$ has dimension $\sum_{k=1}^iN-k+1 = i(N+1)-\frac{i(i+1)}{2}$. 
\end{proof}

\begin{Corollary}\label{Grass1n4}
Let $\mathbb{G}(1,n)\subset\mathbb{P}^N$ be the Grassmannian of lines in $\mathbb{P}^n$ embedded via the Pl\"ucker embedding, 
and let $X_i = \mathbb{G}(1,n)\cap H_1\cap ...\cap H_i$ be a codimension $i$ smooth linear section of $\mathbb{G}(1,n)$, with $0\leq i \leq 2(n-1)-4$.
Then the restriction map 
$$
r_i:H^0(X_{i-1},\Omega_{X_{i-1}}^1(2))\rightarrow H^0(X_i,\Omega_{X_i}^1(2))
$$
corresponds to the linear projection
$$
\pi_i:\mathbb{P}(H^0(X_{i-1},\Omega_{X_{i-1}}^1(2)))\dashrightarrow \mathbb{P}(H^0(X_i,\Omega_{X_i}^1(2)))
$$
with center $L\cong \mathbb{P}(H^0(X_i,\mathcal{O}_{X_i}(1)))$. 

Furthermore, $L$ is the subspace of $\mathbb{P}(H^0(X_{i-1},\Omega_{X_{i-1}}^1(2)))$ parametrizing integrable $1$-forms corresponding to the restriction to $X_{i-1}\subset H_1\cap ...\cap H_{i-1}\cong\mathbb{P}^{N-i+1}$ of a linear projection $\mathbb{P}^{N-i+1}\dasharrow\mathbb{P}^1$ from a codimension two linear subspace contained in $H_i$.
\end{Corollary}

\begin{proof}
The codimension two linear subspaces of $H_1\cap ...\cap H_{i-1}\cong\mathbb{P}^{N-i+1}$ contained in $H_i$ form a vector space of dimension $N-i+1$ contained in $\ker(r_i)\cong H^0(X_i,\mathcal{O}_{X_i}(1))$. 
On the other hand, we saw in the  proof of Lemma \ref{Grass1n3} that  $h^0(X_i,\mathcal{O}_{X_i}(1)) = N-i+1$. 
\end{proof}

Now we specialize to the Grassmannian $\mathbb{G}(1,4)\subset\mathbb{P}^9$, and determine what happens to the class of a 
distribution on $\mathbb{P}^9$ under restriction to $\mathbb{G}(1,4)$.

\begin{Proposition}\label{PropG14}
Let $r:H^0(\mathbb{P}^9,\Omega_{\mathbb{P}^9}^1(2))\rightarrow H^0(\mathbb{G}(1,4),\Omega_{\mathbb{G}(1,4)}^1(2))$ 
be the restriction isomorphism, and consider a twisted $1$-form  $\omega\in H^0(\mathbb{P}^9,\Omega_{\mathbb{P}^9}^1(2))$. Then 
\begin{enumerate}
\item If $r(\omega)$ has class zero, then $\omega$ has class zero.
\item If $r(\omega)$ has class one, then one of the following holds:
\begin{itemize}
\item[-]  $\omega$ has class one, or
\item[-] $\omega$ has class two, and the characteristic foliation of $\omega$, induced by $\omega\wedge (d\omega)^2$, 
is the linear projection $\p^9\dasharrow \p^5$ from a $3$-dimensional linear subspace contained in $\mathbb{G}(1,4)$.
\end{itemize}
\end{enumerate}
\end{Proposition}

\begin{proof}
Clearly the class of $\omega$ can only drop  under restriction to $\mathbb{G}(1,4)$.
By \cite[Theorem 5]{AD4} any foliation in $H^0(\mathbb{G}(1,4),\Omega_{\mathbb{G}(1,4)}^1(2))$ is induced by a  foliation on $\mathbb{P}^9$. This proves the first statement. 

Let $\omega\in H^0(\mathbb{P}^9,\Omega_{\mathbb{P}^9}^1(2))$ be distribution of class two, that is $\omega\wedge (d\omega)^2\neq 0$, but $\omega\wedge (d\omega)^3= 0$. Then the non-zero $5$-form $\omega\wedge (d\omega)^2$ induces a codimension five foliation on $\mathbb{P}^9$. Such a foliation is given by the linear projection $\pi_H:\mathbb{P}^9\dasharrow\mathbb{P}^5$ from a linear subspace $H\subset\mathbb{P}^9$ of dimension three. The restriction $r(\omega\wedge (d\omega)^2)$ is zero in $\mathbb{G}(1,4)$ if and only if one of the following holds
\begin{itemize}
\item[-] $\mathbb{G}(1,4)\subseteq \Sing(\omega\wedge (d\omega)^2)$, or
\item[-] the restriction $\pi_{H|\mathbb{G}(1,4)}:\mathbb{G}(1,4)\dasharrow\mathbb{P}^5$ is not dominant.
\end{itemize} 
The first case cannot happen because the singular locus of $\omega\wedge (d\omega)^2$  is a linear subspace of $\mathbb{P}^9$,
while  $\mathbb{G}(1,4)$ is non degenerate.
By Proposition~\ref{projG14} below, the restriction $\pi_{H|\mathbb{G}(1,4)}:\mathbb{G}(1,4)\dasharrow\mathbb{P}^5$ is not dominant
if and only if $H\subset\mathbb{G}(1,4)$. 

It remains to show that if $\omega$ has class $\geq 3$, then $r(\omega)$ has class $2$. With this purpose, 
 consider the local parametrization $\phi: \mathbb{A}^6 \rightarrow  \mathbb{A}^9$ of $\mathbb{G}(1,4)\subset\mathbb{P}^9$, given by 
$$\phi(u_0,u_1,u_2,v_0,v_1,v_2)=  (v_0,v_1,v_2,-u_0,-u_1,-u_2,u_0v_1-u_1v_0,u_0v_2-u_2v_0,u_1v_2-u_2v_1).$$
Let $\omega = \sum_{0\leq i<j\leq 9}a_{ij}(z_idz_j-z_jdz_i)$ be a general $1$-form  in $\mathbb{P}^9$, and consider the $5$-form 
$\phi^{*}\omega\wedge (\phi^{*}\omega)^2$ in $\mathbb{A}^6$. A standard Maple computation for the form $\phi^{*}\omega\wedge (\phi^{*}\omega)^2$ shows that among the coefficients of $\phi^{*}\omega\wedge (\phi^{*}\omega)^2$ there are the equations defining the secant variety $\sec_{3}(\mathbb{G}(1,9))\subset\mathbb{P}^{44}$. 
Therefore, by Theorem \ref{mainstrat}, the vanishing $\phi^{*}\omega\wedge (\phi^{*}\omega)^2=0$ forces $\omega\wedge(d\omega)^3=0$ in $\mathbb{P}^9$.
Hence the class of $\omega$ is at most two. 
\end{proof}

\begin{Proposition}\label{projG14}
Let $\pi_H:\mathbb{P}^9\dasharrow\mathbb{P}^5$ be the projection from a linear subspace $H$ of dimension three. 
Then $\pi_{H|\mathbb{G}(1,4)}$ is dominant if and only if $H\not\subset\mathbb{G}(1,4)$.
\end{Proposition}

\begin{proof}
For a point $p\in\mathbb{P}^4$, we denote by $H_p\subset\mathbb{G}(1,4)$ the linear $\p^3$ parametrizing lines through $p$.
Recall that a linear subspace $H\subset \mathbb{P}^9$ of dimension three is contained in $\mathbb{G}(1,4)$ if and only if it is of the form 
$H_p$ for some $p\in\mathbb{P}^4$.

When $H=H_p$, the projection $\pi_{H_p|\mathbb{G}(1,4)}:\mathbb{G}(1,4)\dasharrow\mathbb{P}^5$ is induced by the projection 
$\pi_p:\mathbb{P}^4\dasharrow\mathbb{P}^3$ from $p$. Therefore, 
$\overline{\pi_{H_p|\mathbb{G}(1,4)}(\mathbb{G}(1,4))} = \mathbb{G}(1,3)$ which is a quadric hypersurface in $\mathbb{P}^5$.

Now suppose that $\pi_{H|\mathbb{G}(1,4)}$ is not dominant, and set $X = \overline{\pi_{H_p}(\mathbb{G}(1,4))}$. 
We will show that $H\subset\mathbb{G}(1,4)$.
Given two general points $p,q\in\mathbb{G}(1,4)$ corresponding two  skew lines $L_p, L_q\subset\mathbb{P}^4$,
we let $\mathbb{G}(1,3)_{p,q}\subset\mathbb{G}(1,4)$ be the subvariety parametrizing  lines contained in 
$\left\langle L_p, L_q\right\rangle$.  It is isomorphic to we get a Grassmannian $\mathbb{G}(1,3)$ and generates a linear space 
$\mathbb{P}^5\cong H_{p,q}\subset \mathbb{G}(1,4)$. 

We claim that, for  general points $p,q\in\mathbb{G}(1,4)$, we have $H\cap H_{p,q}=\emptyset$.
Indeed, if $H\cap H_{p,q}\neq \emptyset$, then $\pi_{H|H_{p,q}}:H_{p,q}\dasharrow\mathbb{P}^5$ is a linear projection from a linear subspace 
$\overline{H}_{p,q}\subset H_{p,q}$ with $0\leq \dim(\overline{H}_{p,q})\leq 3$. 
Since $\mathbb{G}(1,3)_{p,q}\subset H_{p,q}$ is a quadric hypersurface, 
we obtain a positive dimensional linear subspace $\overline{\pi_{H}(\mathbb{G}(1,3)_{p,q})}\subseteq X$ through $\pi_{H}(p)$ and $\pi_{H}(q)$. 
Therefore, through two general points of $X$, there is a positive dimensional linear space, and this forces $X$ to be a proper linear subspace of $\mathbb{P}^5$.
This is contradicts the fact that $\mathbb{G}(1,4)$ is non degenerate, proving that $H\cap H_{p,q}=\emptyset$, and thus 
$\pi_{H|H_{p,q}}:H_{p,q}\dasharrow\mathbb{P}^5$ is an isomorphism. 
Since $\pi_{H}(\mathbb{G}(1,3)_{p,q})\subseteq X$, we conclude that $X\subset\mathbb{P}^5$ is a smooth quadric hypersurface.

Let $B \subset H\cap\mathbb{G}(1,4)$ be the indeterminacy locus of $\pi_{H|\mathbb{G}(1,4)}$, 
consider the following incidence variety
 \[
  \begin{tikzpicture}[xscale=1.5,yscale=-1.5]
    \node (A0_1) at (1, 0) {$\mathcal{I} = \{(x,[L])\: | \: x\in L\}\subseteq \mathbb{P}^4\times B$,};
    \node (A1_0) at (0, 1) {$\mathbb{P}^4$};
    \node (A1_2) at (2, 1) {$B$};
    \path (A0_1) edge [->]node [auto] {$\scriptstyle{\psi}$} (A1_2);
    \path (A0_1) edge [->]node [auto,swap] {$\scriptstyle{\phi}$} (A1_0);
  \end{tikzpicture}
  \] 
and note that the fibers of $\psi$ are $1$-dimensional. If $\dim(B)\leq 2$, then $\dim(\phi(\mathcal{I}))\leq 3$. 
Therefore, there is no a line parametrized by $B$ passing through a general point $p\in\mathbb{P}^4$.
In other words, $H_p\cap B = \emptyset$, and so $\pi_{H}(H_p)\subset X$ is a $3$-dimension linear subspace. 
But this is impossible because there no linear subspace of dimension three contained in a smooth quadric hypersurface in $\mathbb{P}^5$.
We conclude that $B=H\subset \mathbb{G}(1,4)$.
\end{proof}

\begin{Remark}
We thank Jos\'e Carlos Sierra for the following alternative proof of Proposition \ref{projG14}.
Again, assuming that $\pi_{H|\mathbb{G}(1,4)}$ is not dominant,  let $X\subset\mathbb{P}^5$ be the image of $\mathbb{G}(1,4)$. 
Let $x\in X$ be a general point, and set $F_x=\pi_{H_p|\mathbb{G}(1,4)}^{-1}(x)$.
Let $\Gamma\subset\mathbb{P}^5$ be  a hyperplane  containing the tangent space $\mathbb{T}_xX$. 
Then $ \pi_{H}^{-1}(\Gamma)$ is a hyperplane tangent to $\mathbb{G}(1,4)$ along $F_x$. 
Recall that a hyperplane in $\mathbb{P}^9$  that is tangent  to $\mathbb{G}(1,4)$ at some point 
must be tangent to $\mathbb{G}(1,4)$ along a $\beta$-plane, that is, a plane parametrizing lines in a plane of $\mathbb{P}^4$. 
This is a manifestation of the self-duality of $\mathbb{G}(1,4)$. 
Therefore, $F_x$ must be a $\beta$-plane $\Pi_x$. Since $F_x$ is contracted to a point via $\pi_{H}$ it must intersect $H$ in a line. 
Such a line parametrizes lines in $\Pi_x$ passing through a fixed point $p\in \Pi_x$. 
The point $p\in \mathbb{P}^4$ does not depend on $x\in X$, and so $H=H_p$. 
\end{Remark}

\end{document}